\documentclass[dvipdfmx,12pt]{amsart}
\usepackage{amsfonts}
\usepackage{amssymb}
\usepackage{amscd}
\usepackage{mathrsfs}
\usepackage{tikz}
\usepackage{hyperref}

\newtheorem{theorem}{Theorem}[section]
\newtheorem{corollary}[theorem]{Corollary}

\newtheorem{lemma}[theorem]{Lemma}
\newtheorem{proposition}[theorem]{Proposition}

\newtheorem{remark}{Remark}[section]

\newtheorem{definition}{Definition}[section]

\newtheorem*{theorem*}{Theorem}
\newtheorem*{corollary*}{Corollary}
\newtheorem*{conjecture*}{Conjecture}
\newtheorem*{lemma*}{Lemma}
\newtheorem*{proposition*}{Proposition}
\newtheorem*{problem*}{Problem}
\newtheorem*{axiom*}{Axiom}
\newtheorem*{remark*}{Remark}
\newtheorem*{example*}{Example}
\newtheorem*{exercise*}{Exercise}
\newtheorem*{definition*}{Definition}

\numberwithin{equation}{section}

\newcommand{\im}{{\rm Im}}
\newcommand{\re}{{\rm Re}}

\renewcommand{\eqref}[1]{(\ref{#1})}

\renewcommand{\bigskip}{\vspace{0.2cm}}
\renewcommand{\l}{\left}
\renewcommand{\r}{\right}

\newcommand{\cleq}{\lesssim}

\newcommand{\norm}{\Vert}

\newcommand{\ceq}{\thickapprox}
\newcommand{\wto}{\rightharpoonup}

\def\eqref[#1]{(\ref{#1})}
\def\secref[#1]{Section~\ref{#1}}
\def\appref[#1]{Appendix~\ref{#1}}
\def\lemref[#1]{Lemma~\ref{#1}}
\def\thmref[#1]{Theorem~\ref{#1}}
\def\corref[#1]{Corollary~\ref{#1}}
\def\propref[#1]{Proposition~\ref{#1}}
\def\conjref[#1]{Conjecture~\ref{#1}}
\def\defref[#1]{Definition~\ref{#1}}
\def\abs[#1]{|#1|}
\def\norm[#1]{\left \Vert #1 \right \Vert}
\def\tbra[#1,#2]{\langle #1 | #2\rangle} 
\def\rbra[#1,#2]{( #1 | #2 )} 
\def\sbra[#1,#2]{[ #1 | #2 ]} 
\def\fbra[#1,#2]{\{ #1 | #2 \}} 
\def\besov[#1,#2,#3]{B_{#2,#3}^{#1}}
\def\hbesov[#1,#2,#3]{\dot{B}_{#2,#3}^{#1}}

\newcommand{\ga}{\gamma}

\newcommand{\eps}{\varepsilon}

\newcommand{\la}{\lambda}

\newcommand{\sech}{{\rm sech}}

\newcommand{\del}{\partial}

\newcommand{\N}{{\mathbb N}}
\newcommand{\R}{{\mathbb R}}

\newcommand{\cC}{{\mathcal C}}
\newcommand{\cD}{{\mathcal D}}

\newcommand{\cH}{{\mathcal H}}

\newcommand{\cJ}{{\mathcal J}}

\newcommand{\cL}{{\mathcal L}}
\newcommand{\cM}{{\mathcal M}}
\newcommand{\cN}{{\mathcal N}}

\newcommand{\cR}{{\mathcal R}}

\begin{document}
\title[Global Dynamics for NLS with a repulsive Dirac delta potential]
{Global Dynamics below the standing waves for the focusing semilinear Schr\"{o}dinger equation with a repulsive Dirac delta potential}
\author[M. Ikeda]{Masahiro Ikeda}
\address{Department of Mathematics\\ Graduate School of Science, Kyoto University, Kyoto\\ Kyoto, 606-8502, Japan}
\email{mikeda@math.kyoto-u.ac.jp}
\author[T. Inui]{Takahisa Inui}
\address{Department of Mathematics\\ Graduate School of Science, Kyoto University, Kyoto\\ Kyoto, 606-8502, Japan}
\email{inui@math.kyoto-u.ac.jp}
\date{}
\maketitle

\begin{abstract}
We consider the focusing mass supercritical semilinear Schr\"{o}dinger equation with a repulsive Dirac delta potential on the real line $\R$:
\begin{equation}
\l\{
\begin{array}{ll}
i \del_t u + \frac{1}{2} \del_{x}^{2} u + \ga \delta_0 u + |u|^{p-1}u=0, & (t,x) \in \R \times \R,
\\
u(0,x)=u_0(x)\in H^1(\R), 
\end{array}
\r.
\tag{$\delta$NLS}
\end{equation}
where $\gamma \leq 0$, $\delta_0$ denotes the Dirac delta with the mass at the origin, and $p>5$. It is known that \eqref[deltaNLS] is locally well-posed in the energy space $H^1(\R)$ and there exist standing wave solutions $e^{i\omega t} Q_{\omega,\gamma}(x)$ when $\omega > \ga^2/2$, where $Q_{\omega,\gamma}$ is a unique radial positive solution to $-\frac{1}{2}\del_x^2 Q +\omega Q - \ga \delta_0 Q =|Q|^{p-1}Q$ (see \cite{FOO}). Our aim in the present paper is to find a necessary and sufficient condition on the data below the standing wave $e^{i\omega t} Q_{\omega,0}$ to determine the global behavior of the solution. 
The similar result for NLS without potential ($\gamma=0$) was obtained by Akahori--Nawa \cite{AN} (see also \cite{FXC}). Our proof of the scattering result is based on the argument of Banica--Visciglia \cite{BV}, who proved all solutions scatter in the defocusing and repulsive case ($\ga<0$)
by the Kenig--Merle method \cite{KM06}. However, the method of Banica--Visciglia \cite{BV} cannot be applicable to our problem because the energy may be negative in the focusing case. To overcome this difficulty, we use the variational argument based on \cite{IMN}. 
Our proof of the blow-up result is based on the method of Du--Wu--Zhang \cite{DWZ}. Moreover, we determine the global dynamics of the radial solution whose mass-energy is larger than that of the standing wave $e^{i\omega t} Q_{\omega,0}$. The difference comes from the existence of the potential.
\end{abstract}

\tableofcontents


\section{Introduction}

\subsection{Background}

We consider the focusing mass supercritical semilinear Schr\"{o}dinger equation with a repulsive Dirac delta potential on the real line $\R$:
\begin{equation}
\l\{
\begin{array}{ll}
i \del_t u + \frac{1}{2} \del_{x}^{2} u + \ga \delta_0  u + |u|^{p-1}u=0, & (t,x) \in \R \times \R,
\\
u(0,x)=u_0(x) \in H^1(\R), &
\end{array}
\r.
\tag{$\delta$NLS}
\label{deltaNLS}
\end{equation}
where $\gamma \leq 0$, $\delta_0$ denotes the Dirac delta with the mass at the origin, and $p>5$. \eqref[deltaNLS] appears in a wide variety of physical models with a point defect on the line \cite{GHW04} and references therein. We define the Schr\"{o}dinger operator $H_{\gamma}$ as the formulation of a formal expression $-\frac{1}{2}\del_x^2 - \gamma \delta_0$.
\begin{align*}
&H_{\gamma} \phi := - \frac{1}{2} \del_x^2 \phi , \ \phi \in \cD(H_{\gamma}),
\\
& \cD(H_{\gamma}):=\{ \phi \in H^1(\R)\cap H^2(\R\setminus \{0\}): \del_x \phi(0+)-\del_x \phi (0-)= -2\ga \phi(0)\}.
\end{align*}
$H_{\ga}$ is a non-negative self-adjoint operator on $L^2(\R)$ (see \cite{AGHH} for more details), which implies that \eqref[deltaNLS] is locally well-posed in the energy space $H^1(\R)$. 

\begin{proposition}[{\cite[Section 2]{FOO}, \cite[Theorem 3.7.1]{C}}]
For any $u_0\in H^1(\R)$, there exist $T_{\pm}=T_{\pm}(\norm[u_0]_{H^1})>0$ and a unique solution $u\in C((-T_{-},T_{+});H^1(\R)) \cap C^1((-T_{-},T_{+});H^{-1}(\R))$ of \eqref[deltaNLS]. Moreover, the following statements hold.
\begin{itemize}
\item (Blow-up criterion) $T_{\pm}=\infty$, or $T_{\pm}<\infty$ and $\lim_{t\to \pm T_{\pm}}\norm[\del_x u (t)]_{L^2}=\infty$. (Double-sign corresponds.)
\item (Coservation Laws) The energy $E$ and the mass $M$ are conserved by the flow, \textit{i.e.} 
\[ E(u(t))=E(u_0),\ M(u(t))=M(u_0), \text{ for any } t\in (-T_{-},T_{+}),\]
where for $\phi \in H^1(\R)$, $E$ and $M$ are defined as
\begin{align}
\label{energy}
E(\phi)&=E_{\ga}(\phi):= \frac{1}{4} \norm[\del_x \phi]_{L^2}^2  -\frac{\ga}{2} | \phi(0)|^2 -\frac{1}{p+1} \norm[\phi]_{L^{p+1}}^{p+1},
\\
\label{mass}
M(\phi)&:=\frac{1}{2}\norm[\phi]_{L^2}^2.
\end{align}
\end{itemize}
\end{proposition}

We investigate the global behaviors of the solution. By the choice of the initial data, \eqref[deltaNLS] has various solutions, for example, scattering solution, blow-up solution, and so on. Let us recall the definitions of scattering and blow-up.
Let $u$ be a solution to \eqref[deltaNLS] on the maximal existence time interval $(-T_{-}, T_{+})$.

\begin{definition}[scattering]
We say that the solution $u$ to \eqref[deltaNLS] scatters if and only if $T_{\pm}=\infty$ and there exist $u_{\pm} \in H^1(\R)$ such that 
\begin{align*}
\norm[ u(t) - e^{-it H_{\ga}} u_{\pm}]_{H^1} \to 0, \text{ as } t \to \pm \infty.
\end{align*}
where $\{e^{-it H_{\ga}}\}$ denotes the evolution group of $i\del_t u -H_{\gamma} u=0$. 
\end{definition}


\begin{definition}[blow-up]
We say that the solution $u$ to \eqref[deltaNLS] blows up in positive time (resp. negative time) if and only if $T_{+}< \infty$ (resp. $T_{-} < \infty$). 
\end{definition}


Since a pioneer work by Kenig and Merle \cite{KM06},  the global dynamics without assuming smallness for  focusing nonlinear Schr\"{o}dinger equations have been studied. For the focusing cubic semilinear Schr\"{o}dinger equation in three dimensions, Holmer and Roudenko \cite{HR08} proved that $\norm[u_0]_{L^2}\norm[\nabla u_0]_{L^2}<\norm[Q]_{L^2}\norm[\nabla Q]_{L^2}$ implies scattering and, on the other hand, $\norm[u_0]_{L^2}\norm[\nabla u_0]_{L^2}>\norm[Q]_{L^2}\norm[\nabla Q]_{L^2}$ implies finite time blow-up  if the initial data $u_0 \in H^1(\R^3)$ is radially symmetric and satisfies the mass-energy condition $M(u_0)E(u_0) < M(Q) E(Q)$, where $Q$ is the ground state. For non-radial solutions, Duyckaerts, Holmer, and Roudenko \cite{DHR} proved the scattering part and Holmer and Roudenko \cite{HR10} proved the solutions in the above blow-up region blow up in finite time or grow up in infinite time. Fang, Xie, and Cazenave \cite{FXC} extend the scattering result and Akahori and Nawa \cite{AN} extend both the scattering and the blow-up result to mass supercritical and energy subcritical Schr\"{o}dinger equations in general dimensions. 

Recently, Banica--Visciglia \cite{BV} proved all solutions scatter in the defocusing case. On the othe hand, in the focusing case, \eqref[deltaNLS] has blow-up solutions and non-scattering global solution. Thus, their method cannot be applicable to our problem.

\subsection{Main Results}

To state our main result, we introduce several notations. 

Let $\omega$ be a positive parameter and $\omega$ denotes the frequency. 
We set action $S_{\omega}$ and a functional $P$ as follows. 
\begin{align}
S_{\omega}(\phi)&=S_{\omega,\ga}(\phi)  : = E(\phi)+\omega M(\phi)
\\
\notag
&=\frac{1}{4} \norm[\del_x \phi]_{L^2}^2 -\frac{\ga}{2} | \phi (0)|^2+\frac{\omega}{2}\norm[\phi]_{L^2}^2  -\frac{1}{p+1} \norm[\phi]_{L^{p+1}}^{p+1},
\\
P(\phi) &=P_{\ga}(\phi) :=\frac{1}{2} \norm[\del_x \phi]_{L^2}^2  -\frac{\ga}{2} | \phi (0)|^2 - \frac{p-1}{2(p+1)} \norm[\phi]_{L^{p+1}}^{p+1},
\end{align}
where $P$ appears in the virial identity (see \cite{LFFKS}).

We often omit the index $\ga$. 
We sometimes insert $0$ into $\ga$, such as $S_{\omega,0}$ and $P_{0}$. 

We consider the following three minimizing problems. 
\begin{align}
n_{\omega}
&:= \inf \{ S_{\omega} (\phi) : \phi\in H^1(\R) \setminus \{0\}, P(\phi)=0\},
\\
r_{\omega}
&:= \inf \{ S_{\omega} (\phi) : \phi\in H_{rad}^1(\R) \setminus \{0\}, P(\phi)=0\},
\\
l_{\omega}
&:= \inf \{ S_{\omega,0} (\phi) : \phi\in H^1(\R) \setminus \{0\}, P_{0}(\phi)=0\},
\end{align}
where $H^1_{rad}(\R):=\{ \varphi \in H^1(\R): \varphi(x)=\varphi(-x)\}$.

$l_{\omega}$ is nothing but the minimizing problem for the nonlinear Schr\"{o}dinger equation without a potential and $l_{\omega}$ is positive and is attained by 
\[ Q_{\omega,0}(x)
:=\l\{  \frac{(p+1)\omega}{2} \sech ^2  \l(  \frac{(p-1)\sqrt{\omega}}{\sqrt{2}}|x| \r) \r\}^{\frac{1}{p-1}}, \]
which is a unique positive solution of 
\begin{equation} \label{SE}
-\frac{1}{2}\del_x ^2 Q + \omega Q =|Q|^{p-1}Q.
\end{equation} 

For $n_{\omega}$ and $r_{\omega}$, we prove the following statements, some of which were proved by Fukuizumi--Jeanjean \cite{FJ}. 

\begin{proposition}
\label{prop:mini}
Let $\gamma$ be strictly negative. Then the following statements are true. 
\begin{enumerate}
\item $n_{\omega} = l_{\omega}$ and $n_{\omega}$ is not attained.
\item $n_{\omega} < r_{\omega}$ and 
\begin{align*}
\l\{
\begin{array}{ll}
r_{\omega} = 2l_{\omega},  & \text{ if } 0< \omega \leq \ga^2/2,
\\
r_{\omega} < 2l_{\omega}, & \text{ if } \omega > \ga^2/2.
\end{array}
\r.
\end{align*}
\item If $\omega > \ga^2/2$, then $r_{\omega}$ is attained by
\[ Q_{\omega}(x)=Q_{\omega,\ga}(x)
:=\l\{  \frac{(p+1)\omega}{2} \sech ^2  \l(  \frac{(p-1)\sqrt{\omega}}{\sqrt{2}}|x|+ \tanh ^{-1} \l( \frac{\ga}{\sqrt{2\omega}} \r)   \r) \r\}^{\frac{1}{p-1}},\]
which is a unique positive solution of $-\frac{1}{2}\del_x^2 Q +\omega Q - \ga \delta_0 Q =|Q|^{p-1}Q.$
\\
On the other hand, $r_{\omega}$ is not attained if $0<\omega \leq \ga^2/2$. 
\end{enumerate}
\end{proposition}

The function $e^{i\omega t}Q_{\omega}$ with $\omega > \ga^2/2$ is a global non-scattering solution to \eqref[deltaNLS], which is called standing wave. The fact that $n_{\omega} \neq r_{\omega}$ comes from the existence of the potential, which means that the following main result in the radial case dose not follow from that in the non-radial case.  

By using the minimizing problems, we define subsets in $H^1(\R)$ for $\omega>0$ as follows.
\begin{align*}
\cN_{\omega}^{+}
&:=\{ \varphi \in H^1(\R):  S_{\omega}(\varphi)<n_{\omega}, P(\varphi)\geq0\},
\\
\cN_{\omega}^{-}
&:=\{ \varphi \in H^1(\R):  S_{\omega}(\varphi)<n_{\omega}, P(\varphi) <0\},
\end{align*}
and 
\begin{align*}
\cR_{\omega}^{+}
&:=\{ \varphi \in H_{rad}^1(\R):  S_{\omega}(\varphi)< r_{\omega}, P(\varphi)\geq0\},
\\
\cR_{\omega}^{-}
&:=\{ \varphi \in H_{rad}^1(\R):  S_{\omega}(\varphi)< r_{\omega}, P(\varphi) <0\}.
\end{align*}

We state one of our main results, which treats  the non-radial case. 
We classify the global behavior of the solution whose action is less than $n_{\omega}$. 

\begin{theorem}[non-radial case]\label{thm:non-radial}
Let $\omega>0$. Let $u$ be a solution to \eqref[deltaNLS]  on $(-T_{-},T_{+})$ with the initial data $u_0\in H^1(\R)$. 
\begin{enumerate}
\item If the initial data $u_0$ belongs to $\cN_{\omega}^{+}$, then the solution $u$ scatters. 
\item If the initial data $u_0$ belongs to $\cN_{\omega}^{-}$, then one of the following four cases holds. 
\begin{enumerate}
\item The solution $u$ blows up in both time directions.
\item The solution $u$ blows up in a positive time, and $u$ is global toward negative time and $\limsup_{t\to -\infty } \norm[\del_x u(t)]_{L^2}=\infty$ holds.
\item The solution $u$ blows up in a negative time, and $u$ is global toward positive time and $\limsup_{t\to \infty } \norm[\del_x u(t)]_{L^2}=\infty$ holds.
\item The solution $u$ is global in both time directions and $\limsup_{t\to \pm \infty } \norm[\del_x u(t)]_{L^2}=\infty$ holds.
\end{enumerate}
\end{enumerate}
\end{theorem}

\propref[prop:mini] and a direct calculation give $n_{\omega}=l_{\omega}=\omega^{\frac{p+3}{2(p-1)}}S_{1,0}(Q_{1,0})$. By these relations, we can rewrite the main theorem in the non-radial case into the version independent of the frequency $\omega$.

\begin{corollary}\label{cor:non-radial}
We define the subsets $\cN^{\pm}$ in $H^1(\R)$.
\begin{align*}
\cN^{+}&:=\{ \varphi \in H^1(\R):  E(\varphi)M(\varphi)^{\sigma}<E_0(Q_{1,0})M(Q_{1,0})^{\sigma}, P(\varphi)\geq0\},
\\
\cN^{-}&:=\{ \varphi \in H^1(\R):  E(\varphi)M(\varphi)^{\sigma}<E_0(Q_{1,0})M(Q_{1,0})^{\sigma}, P(\varphi) < 0\},
\end{align*}
where $\sigma := (p+3)/(p-5)$. 
Let $u$ be a solution to \eqref[deltaNLS]  on $(-T_{-},T_{+})$ with the initial data $u_0\in H^1(\R)$. 
Then, we can prove the same conclusion as in \thmref[thm:non-radial] with a replacement $\cN_{\omega}^{\pm}$ into $\cN^{\pm}$ respectively.
\end{corollary}

The equivalency is proved in Appendix A.  


Next, we state the other main result for radial solutions. 
If we restrict solutions to \eqref[deltaNLS] to radial solutions, then we can classify the global behavior of the radial solutions whose action is larger than $n_\omega$ and less than $r_{\omega}$.

\begin{theorem}[radial case]\label{thm:radial}
Let  $\omega>0$ and $u$ be a solution to \eqref[deltaNLS] with the initial data $u_0\in H_{rad}^1(\R)$. 
Then, we can prove the same conclusion as in \thmref[thm:non-radial] with a replacement $\cN_{\omega}^{\pm}$ into $\cR_{\omega}^{\pm}$ respectively.
\end{theorem}

\begin{remark}
Even if solutions to \eqref[deltaNLS] are restricted to radial ones, the possibility that (b)--(d) (grow-up) occurs cannot be excluded since we consider one spatial dimension. In \cite{LFFKS}, it was proved that if the initial data satisfies $xu_0 \in L^2$ and $P(u_0)<0$, then the solution blows up in a finite time in both time directions. 
\end{remark}


\subsection{Difficulties and Idea for the proofs}
Our proof of the scattering part is based on the argument of Banica--Visciglia  \cite{BV}, where they proved all solutions scatter in the defocusing case. We also use a concentration compactness argument (see Sections 3.3--3.5) and a rigidity argument (see Section 3.5). In the focusing case, it is not clear that each profile has positive energy when we use profile decomposition. To prove this with $\gamma=0$, the orthogonality property of the functional $P_0$ was used in \cite{FXC} and \cite{AN}. However, it is not easy to prove the orthogonality of the functional $P_\ga$ because of the presence of the Dirac delta potential ($\gamma \neq 0$). To overcome this difficulty, we use the Nehari functional $I_{\omega,\ga}$ (see \eqref[eq2.7] for the definition) instead of $P_\gamma$. Then we can prove that the subsets for the data defined by $I_{\omega}$ instead of $P$ are same as the subsets $\cN_{\omega}^{\pm}$ (see \propref[prop2.15])  by the similar argument to \cite{IMN}.  

\thmref[thm:radial] (radial case) does not follow from \thmref[thm:non-radial] (non-radial case) since we treat  solutions whose action is larger than or equal to $n_{\omega}$ in \thmref[thm:radial]. Recently, Killip--Murphy--Visan--Zheng \cite{KMVZ} 
also considered a similar problem and extended the region to classify solutions under radial assumption for NLS with the inverse-square potential.  
They used the radial Sobolev inequality, which is only effective in higher dimensions, to prove a translation parameter in the linear profile decomposition is bounded. However, this method cannot be applied to our problem. 
In one dimensional case, it is not clear whether the translation parameter is bounded or not. To avoid this difficulty, we use the fact that the translation parameter $-x_n$ appears in the profile decomposition if $x_n$ appears (see \corref[cor:LPD] for more detail).

Next, we explain the blow-up results. Holmer and Roudenko \cite{HR10} proved a blow-up result for the cubic  Schr\"{o}dinger equation without potentials in three dimensions by applying the Kenig--Merle method \cite{KM06}. Recently, Du--Wu--Zhang \cite{DWZ} gave a simpler proof for blow-up, in which they only used the localized virial identity. We apply their method to the equation with a potential.



\subsection{Construction of the paper}
In Section 2, we consider the minimizing problems from the viewpoint of variational argument. We prove the existence and non-existence of a minimizer for $r_{\omega}$ and $n_{\omega}$, and that the subsets for the data defined by $I_{\omega}$ instead of $P$ are same as the subsets in $H^1(\R)$ defined by $P$ in this section. 
In Section 3, we prove the scattering results by a concentration compactness argument and a rigidity argument. We explain the necessity of the Nehari functional $I_{\omega}$ instead of $P$. 
In Section 4, we prove the blow-up results, based on the argument of Du--Wu--Zhang \cite{DWZ}.


\section{Minimizing Problems and Variational Structure}


\subsection{Minimizing Problems}

Let $(\alpha,\beta)$ satisfy the following conditions:
\begin{equation} \label{(a,b)}
\alpha > 0, \ 2\alpha-\beta \geq 0, \ 2\alpha+\beta \geq 0, \ (\alpha,\beta )\neq (0,0).
\end{equation}
We set 
\[ \overline{\mu} :=\max \{ 2\alpha-\beta, 2\alpha+\beta\}, \qquad \underline{\mu} :=\min \{ 2\alpha-\beta, 2\alpha+\beta\}. \]
We define a scaling transformation and a derivative of functional as follows:
\begin{align}
\phi^{\alpha,\beta}_{\lambda}(x)&:=e^{\alpha\lambda}\phi(e^{-\beta\lambda}x),
\\
\cL^{\alpha,\beta}_{\lambda_0}S(\phi)&:=\del_{\lambda} S(\phi^{\alpha,\beta}_\lambda) |_{\lambda =\lambda_0},
\\
\cL^{\alpha,\beta}S(\phi)&:=\cL^{\alpha,\beta}_0S(\phi),
\end{align}
for any function $\phi$ and any functional $S:H^1(\R)\to \R$. We define functionals $K_{\omega}^{\alpha,\beta}$ by
\begin{align}
K_{\omega}^{\alpha,\beta}(\phi)
&=K_{\omega,\gamma}^{\alpha,\beta}(\phi)
\\ \notag
&:=\cL^{\alpha,\beta}S_{\omega}(\phi)
\\ \notag
&=\del_{\lambda} S_{\omega}(e^{\alpha\lambda}\phi(e^{-\beta\lambda}\cdot))|_{\lambda=0}
\\ \notag
&=\frac{2\alpha -\beta}{4} \norm[\del_x \phi]_{L^2}^2 +\frac{\omega(2\alpha+\beta)}{2}\norm[\phi]_{L^2}^2 -\ga\alpha | \phi (0)|^2 -\frac{(p+1)\alpha+\beta}{p+1} \norm[\phi]_{L^{p+1}}^{p+1}.
\end{align}
We especially use the following functionals. 
\begin{align}
&P(\phi)=P_{\ga}(\phi) 
:=K_{\omega}^{1/2,-1}(\phi)
=\frac{1}{2} \norm[\del_x \phi]_{L^2}^2  -\frac{\ga}{2} | \phi (0)|^2 - \frac{p-1}{2(p+1)} \norm[\phi]_{L^{p+1}}^{p+1},
\\
&
\label{eq2.7}
I_{\omega}(\phi)=I_{\omega,\ga}(\phi)
:=K_{\omega}^{1,0}(\phi)
=\frac{1}{2} \norm[\del_x \phi]_{L^2}^2 -\ga | \phi (0)|^2 +\omega \norm[\phi]_{L^2}^2  - \norm[\phi]_{L^{p+1}}^{p+1}.
\end{align}

\begin{remark}
Both the functional $P$, which appears in the virial identity \eqref[eq3.3], and the Nehari functional $I_{\omega}$ are used to prove the scattering results. It is proved in \propref[prop2.15] that $P$ and $I_{\omega}$ have same sign under a condition for the action.  To prove this, we introduce the parameter $(\alpha,\beta)$ based on \cite{IMN}. 
\end{remark}

We also use $J_{\omega}^{\alpha,\beta}$ defined by
\begin{align}
J_{\omega}^{\alpha,\beta}(\phi)
=J_{\omega,\gamma}^{\alpha,\beta}(\phi)
:=S_{\omega}(\phi)-K_{\omega}^{\alpha,\beta}(\phi)/{\overline{\mu}}.
\end{align}

\begin{lemma}
We have the following relations.
\begin{align*}
&(\cL^{\alpha,\beta}- \overline{\mu} )\norm[\del_x \phi]_{L^2}^2 =
\l\{
\begin{array}{ll}
0, & \text{ if } \beta \leq 0,
\\
-2 \beta \norm[\del_x^2 \phi]_{L^2}^2, & \text{ if }  \beta >0,
\end{array}
\r.
\\
&(\cL^{\alpha,\beta}- \overline{\mu} )\norm[\phi]_{L^2}^2 = 
\l\{
\begin{array}{ll}
2 \beta \norm[\phi]_{L^2}^2 , & \text{ if } \beta \leq 0,
\\
0, & \text{ if }  \beta >0,
\end{array}
\r.
\\
&(\cL^{\alpha,\beta}- \overline{\mu})|\phi(0)|^2=
\l\{
\begin{array}{ll}
\beta |\phi(0)|^2, & \text{ if } \beta \leq 0,
\\
- \beta |\phi(0)|^2, & \text{ if }  \beta >0,
\end{array}
\r.
\\
&(\cL^{\alpha,\beta}- \overline{\mu})\norm[\phi]_{L^{p+1}}^{p+1}=
\l\{
\begin{array}{ll}
\{(p-1)\alpha +2 \beta \} \norm[\phi]_{L^{p+1}}^{p+1}, & \text{ if } \beta \leq 0,
\\
(p-1)\alpha  \norm[\phi]_{L^{p+1}}^{p+1}, & \text{ if }  \beta >0.
\end{array}
\r.
\end{align*} 
In particular, 
\[ \overline{\mu} J_{\omega}^{\alpha, \beta}
= (\overline{\mu}-\cL^{\alpha,\beta})S_{\omega}(\phi) 
\geq  |\beta| \min \l\{ \frac{1}{2} \norm[\del_x \phi]_{L^2}^2 , \omega \norm[\phi]_{L^2}^2 \r\} - \frac{\ga |\beta| }{2} |\phi(0)|^2 +\frac{(p-5)\alpha}{p+1} \norm[\phi]_{L^{p+1}}^{p+1}. \]
Moreover, we have
\begin{align*}
-(\cL^{\alpha,\beta}-\overline{\mu})(\cL^{\alpha,\beta}- \underline{\mu})S_{\omega}(\phi)
&=(\cL^{\alpha,\beta}- \overline{\mu})(\cL^{\alpha,\beta}-  \underline{\mu}) \l( \frac{\ga}{2}|\phi(0)|^2 +\frac{\norm[\phi]_{L^{p+1}}^{p+1}}{p+1}\r)
\\
&\geq  - \frac{\ga|\beta|^2}{2}|\phi(0)|^2+ \frac{(p-5)\alpha}{p+1} \cL^{\alpha,\beta}\norm[\phi]_{L^{p+1}}^{p+1}
\geq \frac{(p-5)\alpha \overline{\mu}}{p+1} \norm[\phi]_{L^{p+1}}^{p+1}.
\end{align*}
\end{lemma}

\begin{proof}
These relations are obtained by simple calculations. We only note that 
\[ (p-1)\alpha+2\beta=(p-5)\alpha +2(2\alpha+\beta) \geq (p-5) \alpha \]
holds. 
\end{proof}

By this lemma and $p>5$, we find that $J_{\omega}^{\alpha,\beta}(\phi)\geq 0$ for any $\phi \in H^1(\R)$.  Next, we see that $K_{\omega}^{\alpha,\beta}$ is positive near the origin in $H^1(\R)$.

\begin{lemma} \label{lem:positivity}
Let $\{ \phi_n\}_{n\in \N} \subset H^1(\R)\setminus \{ 0\} $ be bounded in $L^2(\R)$ such that $\norm[\del_x \phi_n]_{L^2}\to 0$ as $n \to \infty.$ Then $K_{\omega}^{\alpha,\beta}(\phi_n)>0$ for large $n\in \N$.
\end{lemma}

\begin{proof}
By $\ga<0$, $p>5$, and the Gagliardo-Nirenberg inequality, we have
\begin{align*}
K_{\omega}^{\alpha,\beta}(\phi_n)
&\geq 
\frac{2\alpha -\beta}{4} \norm[\del_x \phi_n]_{L^2}^2 -\frac{(p+1)\alpha+\beta}{p+1} C \norm[\del_x \phi_n]_{L^2}^{(p-1)/2} \norm[\phi_n]_{L^2}^{(p+3)/2} 
>0,
\end{align*}
for sufficiently large $n\in \N$, where $C$ is a positive constant.
\end{proof}

We define the following minimizing problems for $\omega>0$ and $(\alpha,\beta)$ satisfying \eqref[(a,b)]:
\begin{align}
n_{\omega}^{\alpha,\beta}
&:= \inf \{ S_{\omega} (\phi) : \phi\in H^1(\R) \setminus \{0\}, K_{\omega}^{\alpha,\beta}(\phi)=0\},
\\
r_{\omega}^{\alpha,\beta}
&:= \inf \{ S_{\omega} (\phi) : \phi\in H_{rad}^1(\R) \setminus \{0\}, K_{\omega}^{\alpha,\beta}(\phi)=0\},
\\
l_{\omega}^{\alpha,\beta}
&:= \inf \{ S_{\omega,0} (\phi) : \phi\in H^1(\R) \setminus \{0\}, K_{\omega,0}^{\alpha,\beta}(\phi)=0\}.
\end{align}

If $(\alpha,\beta)=(1/2,-1)$, these are nothing but $n_{\omega}$, $r_{\omega}$, and $l_{\omega}$. 
We prove that these minimizing problems are independent of $(\alpha,\beta)$ and \propref[prop:mini] holds in the following subsections. 

\subsection{Radial minimizing problem} \label{sec2.2}

At first, we consider the radial minimizing problem $r_{\omega}^{\alpha,\beta}$. 
For $\ga \leq 0$, $S_{\omega}: H_{rad}^1(\R) \to \R$ satisfies the following Mountain Pass structure.\begin{enumerate}
\item $S_{\omega}(0)=0$.
\item There exist $\delta, \rho>0$ such that $S_{\omega}(\varphi)>\delta$ for all $\varphi$ with $\norm[\varphi]_{H^1}=\rho$.
\item There exists $\psi \in H_{rad}^1(\R)$ such that $S_{\omega}(\psi)<0$ and $\norm[\psi]_{H^1} > \rho$. 
\end{enumerate}
Indeed, (1) is trivial, (2) can be proved by the Gagliardo--Nirenberg inequality, and (3) is obtained by a scaling argument. 

Let 
\begin{align*} 
& \cC :=\{c\in C([0,1]:H_{rad}^1(\R)): c(0)=0, S_{\omega}(c(1))<0 \}, 
\\
& b:= \inf_{c \in \cC} \max_{t \in [0,1]}S_{\omega}(c(t)). 
\end{align*}

\begin{lemma} 
The identity $b=r_{\omega}^{\alpha,\beta}$ holds . 
\end{lemma}

\begin{proof}
At first, we prove $b \leq r_{\omega}^{\alpha,\beta}$. To see this, it is sufficient to prove the existence of $\{ c_n\}\subset \cC$ such that $\max_{t \in [0,1]}S_{\omega}(c_n(t)) \to r_{\omega}^{\alpha,\beta}$ as $n\to \infty$. 
We take a minimizing sequence $\{ \varphi_n\}$ for $r_{\omega}^{\alpha,\beta}$, namely, 
\[ S_{\omega}(\varphi_n) \to r_{\omega}^{\alpha,\beta} \text{ as } n \to \infty \text{ and } K_{\omega}^{\alpha,\beta}(\varphi_n) =0 \text{ for all } n \in \N.  \] 
We set $\tilde{c}_n(\la):=\cL_{\la}^{\alpha,\beta} \varphi_n $ for $\la \in \R$. Then, we see that $S_{\omega}(\tilde{c}_{n}(\la)) <0$ for large $\la$. Moreover, $\max_{\la \in \R} S_{\omega}(\tilde{c}_n (\la)) = S_{\omega}(\tilde{c}_n (0))=S_{\omega}(\varphi_n) \to r_{\omega}^{\alpha,\beta}$ as $n \to \infty$ since $K_{\omega}^{\alpha,\beta}(\varphi_n) =0$ for all $n \in \N$. We define $c'_n(t)$ for $t \in [-L,L]$ such that
\begin{align*}
c'_n(t) := \l\{
\begin{array}{ll}
\tilde{c}_n (t), & \text{if } -\frac{L}{2} \leq t \leq L,
\\
\{ \frac{2}{L} (t+L)\}^{M} \tilde{c}_n (-\frac{L}{2}), & \text{if } -L \leq t < -\frac{L}{2}.
\end{array}
\r.
\end{align*}
$c'$ is continuous in $H^1(\R)$ and we have $S_{\omega}(c'_{n}(L)) <0$ and $\max_{t\in [-L,L]} S_{\omega}(c'_n (t))=S_{\omega}(\varphi_n)\to r_{\omega}^{\alpha,\beta}$ when $L>0$ and $M=M(n)$ are sufficiently large. By changing variables, we obtain a desired sequence $c_n \in \cC$. 
At second, we prove $b \geq r_{\omega}^{\alpha,\beta}$. It is sufficient to prove 
\[ c([0,1]) \cap \{ \varphi \in H_{rad}^1(\R) \setminus \{0\}: K_{\omega}^{\alpha,\beta}(\varphi)=0\} \neq \emptyset \text{ for all } c \in \cC.  \]
We take arbitrary $c \in \cC$. Now, $c(0)=0$ and $S_{\omega}(c(1))<0$. Therefore, $K_{\omega}^{\alpha,\beta}(c(t))>0$ for some $t \in (0,1)$ by \lemref[lem:positivity] and $K_{\omega}^{\alpha,\beta}(c(1)) \leq \{ (p+1)\alpha+\beta\} S_{\omega}(c(1)) <0$. By the continuity, there exists $t_0 \in (0,1)$ such that $K_{\omega}^{\alpha,\beta}(c(t_0)) =0$. 
Thus, we get $b=r_{\omega}^{\alpha,\beta}$. 
\end{proof}


Next, we prove the existence and non-existence of a minimizer for the minimizing problem $r_{\omega}^{\alpha,\beta}$. See \cite[Lemma 15, 19, 20, 21, and 25]{FJ} for the proofs of the following Lemma 2.4, 2.5, 2.6, 2.7, and 2.8, respectively. 

The following lemma means that it is sufficient to find a non-negative minimizer.

\begin{lemma}
If $\varphi \in H^1(\R)$ is a minimizer of $r_{\omega}^{\alpha,\beta}$, then $|\varphi| \in H^1(\R)$ is also a minimizer.
\end{lemma}


We define a Palais--Smale sequence. 

\begin{definition}[Palais--Smale sequence]
We say that $\{\varphi_n \}_{n \in \N} \subset H^1(\R)$ is a Palais--Smale sequence for $S_{\omega}$ at the level $c$ if and only if the sequence $\{\varphi_n\}_{n\in \N}$ satisfies 
\[ S_{\omega}(\varphi_n) \to c, \text{ and } S_{\omega}'(\varphi_n ) \to 0 \text{ in } H^{-1}(\R), \text{ as } n \to \infty. \]
\end{definition}

By the Mountain Pass theorem, we obtain a Palais--Smale sequence at the level $b=r_{\omega}^{\alpha,\beta}$.
We may assume that the sequence is bounded. 

\begin{lemma}
Any Palais--Smale sequence of $S_{\omega}$ considered on $H_{rad}^1(\R)$ is also a Palais--Smale sequence of $S_{\omega}$ considered on $H^1(\R)$. In particular, a critical point of $S_{\omega}$ considered on $H_{rad}^1(\R)$ is also a critical point of $S_{\omega}$ considered on $H^1(\R)$.
\end{lemma}


\begin{lemma}
Let $\{\varphi_n \}_{n\in \N} \subset H^1(\R)$ be a bounded Palais--Smale sequence at the level $c$ for $S_{\omega}$. Then there exists a subsequence still denoted by $\{ \varphi_n \}$ for which the following holds:
there exist a critical point $\varphi_0$ of $S_{\omega}$, an integer $k\geq0$, for $j=1,\cdots,k$, a sequence of points $\{ x_n^j\} \subset \R$, and nontrivial solutions $\nu^j (x)$ of the equation \eqref[SE] satisfying
\begin{align*}
& \varphi_n \wto \varphi_0 \text{ weakly in } H^1(\R),
\\
& S_{\omega} (\varphi_n) \to c=S_{\omega}(\varphi_0)+\sum_{j=1}^{k} S_{\omega,0}(\nu^j),
\\
& \varphi_n - \l(\varphi_0 +\sum_{j=1}^{k} \nu ^j (x-x_n^j)\r) \to 0 \text{ strongly in }H^1(\R),
\\
& |x_n^j|\to \infty , \quad |x_n^j-x_n^i| \to \infty \text{ for } 1\leq j\neq i \leq k,
\end{align*}
as $n \to \infty$, where we agree that in the case $k=0$, the above holds without $\nu^j$ and $x_n^j$.
\end{lemma}


\begin{lemma} \label{lem2.7}
Assume that 
\[ r_{\omega}^{\alpha,\beta} < 2 l_{\omega}^{\alpha,\beta}. \]
Then the bounded Palais--Smale sequence at the level $r_{\omega}^{\alpha,\beta}$ admits a strongly convergent subsequence.   
\end{lemma}


\begin{lemma} \label{lem2.8}
If $\varphi \in H^1(\R)\setminus \{0\}$ is a critical point of $S_{\omega}$, that is, $\varphi$ satisfies 
\begin{equation} \label{2.12}
-\frac{1}{2}\del_x^2 \varphi +\omega \varphi -\ga \delta_0 \varphi =|\varphi|^{p-1}\varphi
\end{equation} 
in the distribution sense, then it satisfies
\begin{align*}
& \varphi \in C^{j}(\R \setminus \{0\}) \cap C(\R), \quad j=1,2,
\\
& -\frac{1}{2} \del_x^2 \varphi + \omega \varphi = |\varphi|^{p-1}\varphi, \quad x \neq 0,
\\
& \del_x \varphi (0+) - \del_x \varphi (0-) = -2\ga \varphi(0),
\\
&\del_x \varphi (x), \varphi (x) \to 0, \text{ as } |x| \to \infty. 
\end{align*}
\end{lemma}


\begin{lemma}\label{lem2.9}
There exists a unique positive classical solution $\varphi$ of \eqref[2.12] if and only if $\omega> \ga^2/2$. It is nothing but $Q_{\omega}$. If $0< \omega \leq \ga^2/2$, then the classical solution does not exist. 
\end{lemma}

\begin{proof}
We have a unique positive classical solution $Q_{\omega,0}$ of \eqref[SE]. If $\omega > \ga^2/2$, then we get a classical solution $\varphi$ of \eqref[2.12] by the translation of $Q_{\omega,0}$. See \cite{FJ} for more detail. 
\end{proof}

\begin{lemma} \label{lem2.10}
The inequality $r_{\omega}^{\alpha,\beta}< 2l_{\omega}^{\alpha,\beta}$ holds when $\omega> \ga^2/2$. 
\end{lemma}

\begin{proof}
When $\omega> \ga^2/2$, $Q_{\omega}$ is well defined. We find that $Q_{\omega}$ satisfies $K_{\omega}^{\alpha,\beta}(Q_{\omega})=0$ and $S_{\omega}(Q_{\omega}) < 2l_{\omega}^{\alpha,\beta}$ by direct calculations.
\end{proof}

By Lemma \ref{lem2.7} and \ref{lem2.10}, we find that when $\omega> \ga^2/2$, the function $Q_{\omega}$ attains $r_{\omega}^{\alpha,\beta}$.

\begin{lemma}
If $0< \omega \leq \ga^2/2$, then $r_{\omega}^{\alpha,\beta} = 2l_{\omega}^{\alpha,\beta}$ holds.
\end{lemma}

\begin{proof}
Suppose that $r_{\omega}^{\alpha,\beta} < 2l_{\omega}^{\alpha,\beta}$. By Lemma \ref{lem2.7} and \ref{lem2.8}, we have a unique positive classical solution of \eqref[2.12], which contradicts \lemref[lem2.9].
Thus, it suffices to show $r_{\omega}^{\alpha,\beta} \leq 2 l_{\omega}^{\alpha,\beta}$ for all $\omega>0$.  Let 
\[ \varphi_n (x) :=
Q_{\omega,0}(x-n) + Q_{\omega,0}(x+n).
\]
Then, $S_{\omega}(\varphi_n) \to 2 l_{\omega}$ and $K_{\omega}^{\alpha,\beta}(\varphi_n) \to 0$ as $n \to \infty$. Thus, there exists a sequence $\{\la_n\}$ such that $K_{\omega}^{\alpha,\beta}(\la_n \varphi_n) = 0$ and $\la_n \to 1$ as $n \to \infty$. Therefore, we have $S_{\omega}(\la_n \varphi_n) \to 2 l_{\omega}$ as $n \to \infty$ and $K_{\omega}^{\alpha,\beta}(\la_n \varphi_n) = 0$ for all $n \in \N$. This means that $r_{\omega}^{\alpha,\beta} \leq 2l_{\omega}^{\alpha,\beta}$. 
\end{proof}

\begin{remark}
The rearrangement argument  implies
\[ l_{\omega}^{\alpha,\beta}=  \inf \{ S_{\omega,0} (\phi) : \phi\in H_{rad}^1(\R) \setminus \{0\}, K_{\omega,0}^{\alpha,\beta}(\phi)=0\}. \]
Therefore, the arguments in Section \ref{sec2.2} do work for $l_{\omega}^{\alpha,\beta}$.
\end{remark}

\subsection{Non-radial minimizing problem} \label{sec2.3}

In this subsection, we prove $n_{\omega}^{\alpha,\beta}=l_{\omega}^{\alpha,\beta}$ and $n_{\omega}^{\alpha,\beta}$ is not attained. 

\begin{lemma}
\label{lem:2.12}
We have 
\[ l_{\omega}^{\alpha,\beta}= j_{\omega}^{\alpha,\beta} := \inf \{ J_{\omega,0}^{\alpha,\beta}(\phi): \phi \in H^1(\R) \setminus \{0\}, K_{\omega,0}^{\alpha,\beta}(\phi) \leq 0 \}. \]
\end{lemma}

\begin{proof}
At first, we prove $j_{\omega}^{\alpha,\beta} \leq l_{\omega}^{\alpha,\beta}$. 
\begin{align*}
j_{\omega}^{\alpha,\beta} 
& \leq \inf \{ J_{\omega,0}^{\alpha,\beta}(\phi): \phi \in H^1(\R) \setminus \{ 0\}, K_{\omega,0}^{\alpha,\beta}(\phi)= 0\}
\\
& =\inf \{ S_{\omega,0}(\phi): \phi \in H^1(\R) \setminus \{ 0\}, K_{\omega,0}^{\alpha,\beta}(\phi) = 0\}
\\
& = l_{\omega}^{\alpha,\beta}.
\end{align*}
Next, we prove $ l_{\omega}^{\alpha,\beta} \leq j_{\omega}^{\alpha,\beta} $. We take $\phi \in H^1(\R) \setminus \{ 0\}$ such that $K_{\omega,0}^{\alpha,\beta}(\phi)\leq 0$. If $K_{\omega,0}^{\alpha,\beta}(\phi)= 0$, then 
\[ l_{\omega}^{\alpha,\beta} \leq S_{\omega,0}(\phi)=J_{\omega,0}^{\alpha,\beta}(\phi). \]
If $K_{\omega,0}^{\alpha,\beta}(\phi)< 0$, then there exists $\lambda_{*} \in (0,1)$ such that $K_{\omega,0}^{\alpha,\beta}(\la_{*} \phi)=0$. Indeed, this follows from the continuity and the fact that $K_{\omega,0}^{\alpha,\beta}(\lambda \phi)> 0$ holds for small $\lambda\in (0,1)$ by \lemref[lem:positivity]. By $\la_* <1$, 
\[ l_{\omega}^{\alpha,\beta} \leq S_{\omega,0}( \lambda_{*} \phi) =J_{\omega,0}^{\alpha,\beta}( \lambda_{*} \phi) \leq J_{\omega,0}^{\alpha,\beta}(\phi). \]
Therefore, we have $l_{\omega}^{\alpha,\beta} \leq J_{\omega,0}^{\alpha,\beta}(\phi)$ for any $\phi \in H^1(\R) \setminus \{ 0\}$ such that $K_{\omega,0}^{\alpha,\beta}(\phi)\leq 0$. This implies $l_{\omega}^{\alpha,\beta} \leq  j_{\omega}^{\alpha,\beta}$. Hence, we get $l_{\omega}^{\alpha,\beta}= j_{\omega}^{\alpha,\beta}$.
\end{proof}

Let $\tau_y \varphi(x):=\varphi(x-y)$ throughout this paper.

\begin{proposition}
The identity $n_{\omega}^{\alpha, \beta} = l_{\omega}^{\alpha, \beta} $ holds. 
\end{proposition}

\begin{proof}
At first, we prove $n_{\omega}^{\alpha,\beta} \geq l_{\omega}^{\alpha,\beta} $. We take arbitrary $\phi \in H^1(\R)\setminus \{0\}$ such that $K_{\omega}^{\alpha, \beta}(\phi) =0$. Since $K_{\omega,0}^{\alpha, \beta}(\phi) \leq K_{\omega}^{\alpha, \beta}(\phi) =0$ due to $\ga \leq 0$, by \lemref[lem:2.12], then we have
\[ l_{\omega}^{\alpha,\beta} \leq J_{\omega,0}^{\alpha,\beta}(\phi) \leq J_{\omega}^{\alpha,\beta}(\phi), \]
which implies
\begin{align*}
l_{\omega}^{\alpha,\beta}
&\leq  \inf \{ J_{\omega}^{\alpha,\beta}(\phi): \phi \in H^1(\R) \setminus \{0\}, K_{\omega}^{\alpha,\beta}(\phi) =0 \}
\\
&=  \inf \{ S_{\omega}(\phi): \phi \in H^1(\R) \setminus \{0\}, K_{\omega}^{\alpha,\beta}(\phi) =0 \}
\\
&= n_{\omega}^{\alpha,\beta}.
\end{align*}
Next, we prove $n_{\omega}^{\alpha,\beta} \leq l_{\omega}^{\alpha,\beta}$. We note that $Q_{\omega,0}$ attains $l_{\omega}^{\alpha,\beta}$. Then, there exists a sequence $\{ y_n\}_{n \in \N}$ with $y_n \to \infty$ as $n \to \infty$ such that $S_{\omega}(\tau_{y_n}Q_{\omega,0})\to S_{\omega,0}(Q_{\omega,0})=l_{\omega}^{\alpha,\beta}$ as $n \to \infty$. 
For this $\{ y_n\} $, $K_{\omega}^{\alpha,\beta}(\tau_{y_n} Q_{\omega,0}) \geq K_{\omega,0}^{\alpha,\beta}(\tau_{y_n} Q_{\omega,0}) =K_{\omega,0}^{\alpha,\beta}(Q_{\omega,0}) =0$ holds for all $n \in \N$. Since $K_{\omega}^{\alpha,\beta}(\lambda \tau_{y_n} Q_{\omega,0}) <0$ for large $\lambda >1$ and $K_{\omega}^{\alpha,\beta}(\tau_{y_n} Q_{\omega,0}) >0$, there exists $\lambda_n >1$ such that $K_{\omega}^{\alpha,\beta}(\lambda_n \tau_{y_n} Q_{\omega,0}) =0$ by the continuity. For this $\{ \lambda_n\}$, we have $\lambda_n \to 1$ as $n \to \infty$. Indeed, since
\begin{align*}
0&=K_{\omega}^{\alpha,\beta}(\lambda_n \tau_{y_n} Q_{\omega,0}) 
\\
&=
\lambda_n ^2
\l( 
\frac{2\alpha -\beta}{4} \norm[\del_x \tau_{y_n} Q_{\omega,0}]_{L^2}^2 
+\frac{\omega(2\alpha+\beta)}{2}\norm[\tau_{y_n} Q_{\omega,0}]_{L^2}^2 
-\ga\alpha | \tau_{y_n} Q_{\omega,0} (0)|^2 
\r)
\\
& \quad
-\lambda_n^{p+1} \frac{(p+1)\alpha+\beta}{p+1} \norm[\tau_{y_n} Q_{\omega,0}]_{L^{p+1}}^{p+1},
\end{align*}
and $K_{\omega,0}^{\alpha,\beta}(\tau_{y_n} Q_{\omega,0})=0$, we have
\begin{align*}
0&=
\frac{2\alpha -\beta}{4} \norm[\del_x \tau_{y_n} Q_{\omega,0}]_{L^2}^2 
+\frac{\omega(2\alpha+\beta)}{2}\norm[\tau_{y_n} Q_{\omega,0}]_{L^2}^2 
-\ga\alpha | \tau_{y_n} Q_{\omega,0} (0)|^2 
\\
& \quad
-\lambda_n^{p-1} \frac{(p+1)\alpha+\beta}{p+1} \norm[\tau_{y_n} Q_{\omega,0}]_{L^{p+1}}^{p+1}
\\
&=
(1-\lambda_n^{p-1} )\frac{(p+1)\alpha+\beta}{p+1} \norm[\tau_{y_n} Q_{\omega,0}]_{L^{p+1}}^{p+1}-\ga\alpha | \tau_{y_n} Q_{\omega,0} (0)|^2 
\\
&=
(1-\lambda_n^{p-1} )\frac{(p+1)\alpha+\beta}{p+1} \norm[Q_{\omega,0}]_{L^{p+1}}^{p+1}-\ga\alpha | \tau_{y_n} Q_{\omega,0} (0)|^2 .
\end{align*}
Therefore, $\lambda_n \to 1$, since  $| \tau_{y_n} Q_{\omega,0} (0)| \to 0$ as $n \to \infty$. Hence, 
$S_{\omega}(\lambda_n \tau_{y_n}Q_{\omega,0}) \to S_{\omega,0}(Q_{\omega,0})=l_{\omega}^{\alpha,\beta}$ as $n \to \infty$ and $K_{\omega}^{\alpha,\beta}(\lambda_n \tau_{y_n} Q_{\omega,0}) =0$ for all $n \in \N$. This implies $n_{\omega}^{\alpha,\beta}\leq l_{\omega}^{\alpha,\beta}$.  
\end{proof}

\begin{proposition}
For any $\omega >0$, 
$n_{\omega}^{\alpha,\beta}$ is not attained, namely, there does not exist $\varphi \in H^1(\R)$ such that $K_{\omega}^{\alpha,\beta}(\varphi)=0$ and $S_{\omega}(\varphi)=n_{\omega}^{\alpha,\beta}$. 
\end{proposition}

\begin{proof}
We assume that $\varphi$ attains $n_{\omega}^{\alpha,\beta}$. If $\varphi(0)=0$, then $S_{\omega,0}(\varphi)=S_{\omega}(\varphi)=n_{\omega}^{\alpha,\beta}=l_{\omega}^{\alpha,\beta}$ and $K_{\omega,0}^{\alpha,\beta}(\varphi)=K_{\omega}^{\alpha,\beta}(\varphi)=0$ holds, that is, $\varphi$ also attains $l_{\omega}^{\alpha,\beta}$. By the uniqueness of the ground state for $l_{\omega}^{\alpha,\beta}$, $\varphi = Q_{\omega,0}$. However, $Q_{\omega,0}(0) \neq 0$. Therefore, $\varphi(0) \neq 0$. Now, $|\varphi(x)| \to 0$ as $x \to \infty$ since $\varphi \in H^1(\R)$. Hence, $|\varphi(0)|> |\varphi(y)|$ for sufficiently large $|y|$. Thus, 
\[ K_{\omega}^{\alpha,\beta}(\tau_{y} \varphi) < K_{\omega}^{\alpha,\beta}(\varphi) =0. \]
Since $K_{\omega}^{\alpha,\beta}(\lambda \tau_{y} \varphi) >0$ for small $\lambda \in (0,1)$ by \lemref[lem:positivity] and $K_{\omega}^{\alpha,\beta}(\tau_{y} \varphi)\leq 0$, there exists $\lambda_{*} \in (0,1)$ such that $K_{\omega}^{\alpha,\beta}(\lambda_{*} \tau_{y} \varphi)=0$ by the continuity. 
By the definition of $n_{\omega}^{\alpha,\beta}$,
\[ n_{\omega}^{\alpha,\beta} \leq J_{\omega}^{\alpha,\beta}(\lambda_{*} \tau_{y} \varphi) < J_{\omega}^{\alpha,\beta}(\tau_{y} \varphi) < J_{\omega}^{\alpha,\beta}(\varphi)\leq n_{\omega}^{\alpha,\beta}. \]
This is a contradiction. 
\end{proof}

Since $S_{\omega,0}(Q_{\omega,0})=l_{\omega}^{\alpha,\beta}=n_{\omega}^{\alpha,\beta}$, $S_{\omega,\gamma}(Q_{\omega,\gamma})=r_{\omega}^{\alpha,\beta}$ if $\omega > \gamma^2/2$, and $2l_{\omega}^{\alpha,\beta}=r_{\omega}^{\alpha,\beta}$ if $\omega \leq \gamma^2/2$ hold, we find that $r_{\omega}^{\alpha,\beta}$, $l_{\omega}^{\alpha,\beta}$ and $n_{\omega}^{\alpha,\beta}$ are independent of $(\alpha,\beta)$ and so we denote $r_{\omega}^{\alpha,\beta}$, $l_{\omega}^{\alpha,\beta}$ and $n_{\omega}^{\alpha,\beta}$ by $r_{\omega}$, $l_{\omega}$ and $n_{\omega}$ respectively and obtain Proposition \ref{prop:mini}.

\subsection{Variational Structure}

We define subsets $\cN_{\omega}^{\alpha,\beta,\pm}$ and $\cR_{\omega}^{\alpha,\beta,\pm}$ in $H^1(\R)$ such that 
\begin{align*}
\cN_{\omega}^{\alpha,\beta,+}
&:=\{ \varphi \in H^1(\R):  S_{\omega}(\varphi) <n_{\omega}, K_{\omega}^{\alpha,\beta}(\varphi)\geq0\},
\\
\cN_{\omega}^{\alpha,\beta,-}
&:=\{ \varphi\in H^1(\R):  S_{\omega}(\varphi)<n_{\omega}, K_{\omega}^{\alpha,\beta}(\varphi) <0\}.
\\
\cR_{\omega}^{\alpha,\beta,+}
&:=\{ \varphi \in H_{rad}^1(\R):  S_{\omega}(\varphi)<r_{\omega}, K_{\omega}^{\alpha,\beta}(\varphi)\geq0\},
\\
\cR_{\omega}^{\alpha,\beta,-}
&:=\{ \varphi \in H_{rad}^1(\R): S_{\omega}(\varphi)<r_{\omega}, K_{\omega}^{\alpha,\beta}(\varphi) <0\}.
\end{align*}

We note that $\cN_{\omega}^{\pm}=\cN_{\omega}^{1/2,-1,\pm}$ and $\cR_{\omega}^{\pm}=\cR_{\omega}^{1/2,-1,\pm}$. From now on, let $(m_{\omega}, \cM_{\omega}^{\alpha,\beta,\pm})$ denote either $(n_{\omega}, \cN_{\omega}^{\alpha,\beta,\pm})$ or $(r_{\omega}, \cR_{\omega}^{\alpha,\beta,\pm})$. 
The following proposition implies that $P$ and $I_{\omega}$ have same sign if $S_{\omega}<m_{\omega}$. 

\begin{proposition} \label{prop2.15}
For any $(\alpha,\beta)$ satisfying \eqref[(a,b)], $\cM_{\omega}^{\pm}=\cM_{\omega}^{\alpha,\beta,\pm}$. 
\end{proposition}

\begin{proof}
It is easy to check that $\cM_{\omega}^{\alpha,\beta,\pm}$ are open subsets in $H^1(\R)$ because of \lemref[lem:positivity].
Moreover, we have $0 \in \cM_{\omega}^{\alpha,\beta,+}$ and $\cM_{\omega}^{\alpha,\beta,+} \cup \cM_{\omega}^{\alpha,\beta,-}$ is independent of $(\alpha,\beta)$.  And $\cM_{\omega}^{\alpha,\beta,+}$ are connected if $\underline{\mu}>0$. Then $\cM_{\omega}^{\alpha,\beta,+}=\cM_{\omega}^{\alpha',\beta',+}$ for $(\alpha,\beta)\neq (\alpha',\beta')$ such that $2\alpha-\beta>0$, $2\alpha+\beta>0$ and $2\alpha'-\beta'>0$, $2\alpha'+\beta'>0$. Of course, then $\cM_{\omega}^{\alpha,\beta,-}=\cM_{\omega}^{\alpha',\beta',-}$. 

We take $\{(\alpha_n,\beta_n)\}$ satisfying $2\alpha_n-\beta_n>0$ and $2\alpha_n+\beta_n>0$ for all $n\in \N$ and $(\alpha_n,\beta_n)$ converges to some $(\alpha,\beta)$ such that $\underline{\mu}=0$. 
Then $K_{\omega}^{\alpha_n,\beta_n} \to K_{\omega}^{\alpha,\beta}$, and so 
\[ \cM_{\omega}^{\alpha,\beta,\pm} \subset \bigcup_{n\in \N} \cM_{\omega}^{\alpha_n,\beta_n,\pm}.\]
Since each set in the right hand side is independent of $(\alpha,\beta)$, so is the left. 
\end{proof}

Let $\norm[\varphi]_{\cH}^2 := \frac{1}{4}\norm[\del_x \varphi]_{L^2}^2+\frac{\omega}{2} \norm[\varphi]_{L^2}^2 -\frac{\ga}{2}|\varphi(0)|^2$. 

\begin{lemma} 
\label{uniform boundedness}
If $P(\varphi)\geq 0$, then
\begin{align*}
S_{\omega}(\varphi) \leq &\norm[\varphi]_{\cH}^2 \leq \frac{p-1}{p-5} S_{\omega}(\varphi),
\end{align*}
which means that $S_{\omega}(\varphi)$ is equivalent to $\norm[\varphi]_{H^1}^2$.
\end{lemma}

\begin{proof}
The left inequality is trivial. We consider  the right inequality. 
\begin{align*}
0
&\leq 2P(\varphi)
\\
&\leq \norm[\del_x \varphi]_{L^2}^2  -\ga| \varphi(0)|^2 - \frac{p-1}{p+1} \norm[\varphi]_{L^{p+1}}^{p+1}
\\
&= - \frac{p-5}{4}\norm[\del_x \varphi]_{L^2}^2 + \frac{\ga(p-3)}{2} | \varphi(0)|^2+(p-1) E(\varphi)
\\
&\leq - \frac{p-5}{4}\norm[\del_x \varphi]_{L^2}^2 + \frac{\ga(p-5)}{2} | \varphi(0)|^2+(p-1) E(\varphi).
\end{align*}
Therefore, we have
\begin{align*}
\frac{p-5}{4}\norm[\del_x \varphi]_{L^2}^2 - \frac{\ga(p-5)}{2} | \varphi(0)|^2 + \frac{(p-5)\omega}{2} \norm[\varphi]_{L^2}^{2}
&\leq (p-1) E(\varphi)+(p-5)\omega M(\varphi)
\\
& \leq (p-1)(E(\varphi)+\omega M(\varphi)).
\end{align*}
Hence, we obtain
\[ \norm[\varphi]_{\cH}^2 \leq \frac{p-1}{p-5} S_{\omega}(\varphi).  \] 
This completes the proof.
\end{proof}


\begin{lemma}
\label{invariant set}
If $u_0 \in \cM_{\omega}^{+}$, then the corresponding solution $u$ stays in $\cM_{\omega}^{+}$ for all $t\in (-T_{-},T_{+})$. Moreover, If $u_0 \in \cM_{\omega}^{-}$, then the corresponding solution $u$ stays in $\cM_{\omega}^{-}$ for all $t\in (-T_{-},T_{+})$.
\end{lemma}

\begin{proof}
Let $u_0 \in \cM_{\omega}^{+}$. Since the energy and the mass are conserved, $u(t)\in \cM_{\omega}^{+} \cup \cM_{\omega}^{-}$ for all $t\in(-T_{-},T_{+})$. We assume that there exists $t_{**}>0$ such that $u(t_{**})\in \cM_{\omega}^{-}$. By the continuity, there exists $t_*\in (0,t_{**})$ such that $P(u(t_*))=0$ and $P(u(t))<0$ for $t\in (t_*,t_{**}]$. By the definition of $m_\omega$, if $u(t_*)\neq 0$, then
\begin{align*}
m_\omega
>E(u_0)+\omega M(u_0)
=E(u(t_*))+\omega M(u(t_*))
\geq m_\omega.
\end{align*}
This is a contradiction. Thus, $u(t_*)=0$. 
By the uniqueness of solution, $u=0$ for all time. This contradicts $u(t_{**})\in \cM_{\omega}^{-}$. By the same argument, the second statement can be proved.
\end{proof}

\lemref[uniform boundedness] and \lemref[invariant set] imply that all the solutions in $\cM_{\omega}^{+}$ are global in both time directions. 

\begin{proposition}[Uniform bounds on $P$]
\label{prop2.18}
There exists $\delta>0$ such that for any $\varphi \in H^1(\R)$ with $S_{\omega}(\varphi)<m_{\omega}$, we have 
\[ P(\varphi)\geq \min \{ 2(m_{\omega}-S_{\omega}(\varphi),\delta\norm[\varphi]_{\cH}^2\} \text{ or } P(\varphi) \leq -2(m_\omega- S_{\omega}(\varphi)). \]
\end{proposition}

\begin{proof}
We may assume $\varphi\neq 0$. $s(\la):=S_\omega(\varphi^{\la})$ and $n(\la):=\norm[\varphi^{\la}]_{L^{p+1}}^{p+1}$, where $\varphi^\lambda(x):=e^{\lambda/2}\varphi (e^{\lambda}x)$ for $\lambda \in \R$. Then $s(0)=S_\omega(\varphi)$ and $s'(0)=P(\varphi)$. And we have
\begin{align*}
\begin{array}{ll}
s(\la)
=\frac{e^{2\la}}{4}\norm[\del_x \varphi]_{L^2}^2 +\frac{\omega}{2}\norm[\varphi]_{L^2}^2 - \frac{\ga e^\la}{2} |\varphi(0)|^2 -\frac{e^{\frac{p-1}{2}\la}}{p+1} \norm[\varphi]_{L^{p+1}}^{p+1},
&
n(\la)
=e^{\frac{p-1}{2}\la}\norm[\varphi]_{L^{p+1}}^{p+1},
\\
s'(\la)
=\frac{e^{2\la}}{2}\norm[\del_x \varphi]_{L^2}^2 - \frac{\ga e^\la}{2} |\varphi(0)|^2 -\frac{e^{\frac{p-1}{2}\la}(p-1)}{2(p+1)} \norm[\varphi]_{L^{p+1}}^{p+1},
&
n'(\la)
=\frac{e^{\frac{p-1}{2}\la}(p-1)}{2} \norm[\varphi]_{L^{p+1}}^{p+1},
\\
s''(\la)
=e^{2\la}\norm[\del_x \varphi]_{L^2}^2 - \frac{\ga e^\la}{2} |\varphi(0)|^2 -\frac{e^{\frac{p-1}{2}\la}(p-1)^2}{4(p+1)} \norm[\varphi]_{L^{p+1}}^{p+1},
&
n''(\la)
=\frac{e^{\frac{p-1}{2}\la}(p-1)^2}{4} \norm[\varphi]_{L^{p+1}}^{p+1}.
\end{array}
\end{align*}

By an easy calculation, we have
\[ s''= 2 s'+ \frac{\ga}{2} |\phi(0)|^2 -\frac{p-5}{2(p+1)}n' \leq 2s' -\frac{p-5}{2(p+1)} n' \leq 2s'.\] 
Firstly, we consider $P<0$. We have $s'(\lambda)>0$ for sufficiently small $\lambda<0$. Therefore, by the continuity, there exists $\la_0 <0$ such that $s'(\la)<0$ for $\la_0<\la\leq 0$ and $s'(\la_0)=0$.  Integrating the inequality on $[\la_0, 0]$, we have 
\[ s'(0) -s'(\la_0)\leq  2 (s(0)- s(\la_0)).\]
Therefore, we obtain 
\[ P(\varphi) \leq -2 (m_\omega -S_\omega(\varphi)).\] 
Secondly, we consider $P \geq 0$.  
If 
\[ 4P(\varphi) \geq \frac{p-5}{2(p+1)}\cL^{1/2,-1}\norm[\varphi]_{L^{p+1}}^{p+1},\]
then, by adding $\frac{p-5}{2} P(\varphi) \geq\frac{p-5}{2}\ \norm[\varphi]_{\cH}^2-\frac{p-5}{2(p+1)}\cL^{1/2,-1}\norm[\varphi]_{L^{p+1}}^{p+1} $ to the both side, we get
\[
\l\{4+\frac{p-5}{2} \r\}P(\varphi) \geq  \frac{p-5}{2}\norm[\varphi]_{\cH}^2. 
\]
Thus, we get $P(\varphi) \geq \delta \norm[\varphi]_{\cH}^2$. If 
\[ 4P(\varphi) < \frac{p-5}{2(p+1)}\cL^{1/2,-1} \norm[\varphi]_{L^{p+1}}^{p+1}, \]
then 
\begin{equation} 
\label{2.1}
0<4s'< \frac{p-5}{2(p+1)} n',
\end{equation}
at $\la=0$. Moreover, 
\[ s''\leq 4s' -2s' - \frac{p-5}{2(p+1)} n' < -2s' \]
holds at $\la=0$. 
 Now let $\la$ increase. As long as \eqref[2.1] holds and $s'>0$, we have $s''<0$ and so $s'$ decreases and $s$ increases. Since $p>5$, also we have 
\[ n'' \geq 2n' \geq 4n>0,\]
for all $\la\geq 0 $
Hence, \eqref[2.1] is preserved until $s'$ reaches 0. It does reach at finite $\la_1>0$. Integrating $s''<-2s'$ on $[0,\la_1]$, we obtain
\[ s'(\la_1)-s'(0)< -2(s(\la_1)-s(0)).\]
Therefore, by the definition of $m_\omega$, 
\[ P(\varphi) > 2(m_\omega - S_\omega(\varphi)).\]
This completes the proof.
\end{proof}


\section{Proof of the scattering part}

\subsection{Strichartz Estimates and Small Data Scattering}

We recall the Strichartz estimates and a small data scattering result in this subsection. See \cite[Section 3.1 and 3.2]{BV} for the proofs. 
We define the exponents $r$, $a$, and $b$ as follows.
\[ r=p+1, \quad a:=\frac{2(p-1)(p+1)}{p+3}, \quad b:= \frac{2(p-1)(p+1)}{(p-1)^2-(p-1)-4}.\]
Then we have the following estimates.
\begin{lemma}[Strichartz estimates]
We have
\begin{align*}
\norm[e^{-itH_{\ga}} \varphi]_{L_t^a L_x^r} \cleq \norm[\varphi]_{H^1},
\\
\norm[e^{-itH_{\ga}} \varphi]_{L_t^{p-1} L_x^\infty} \cleq \norm[\varphi]_{H^1},
\\
\norm[\int_{0}^{t} e^{-i(t-s)H_{\ga}}F(s)ds ]_{L_t^a L_x^r} \cleq \norm[F]_{L_t^{b'} L_x^{r'} },
\\
\norm[\int_{0}^{t} e^{-i(t-s)H_{\ga}}F(s)ds ]_{L_t^{p-1} L_x^\infty} \cleq \norm[F]_{L_t^{b'} L_x^{r'} },
\end{align*}
where $b'$ denotes the H\"{o}lder conjugate  of $b$, namely, $1/{b'} +1/b=1$.
\end{lemma}

\begin{proposition} \label{prop3.2}
Let the solution $u \in C(\R:H^1(\R))$ to \eqref[deltaNLS] satisfy $u\in L_t^{a}(\R:L_x^{r}(\R))$. Then the solution $u$ scatters. 
\end{proposition}

For the proof \propref[prop3.2], see \cite[Proposition 3.1]{BV}.

The similar statement to \propref[prop3.2] for the following semilinear Schr\"{o}dinger equation without potentials is well known.
\begin{align}
\label{nls}
\tag{NLS}
\l\{
\begin{array}{ll}
i \del_t u +\frac{1}{2} \del_x^2 u +|u|^{p-1}u=0, & (t,x)\in\R \times \R,
\\
u(0,x)=u_0(x) \in H^1(\R), & 
\end{array}
\r.
\end{align}
where $p>5$.

\begin{proposition}[small data scattering] \label{prop:SDS}
Let $\varphi \in H^1(\R)$ and $u$, $v$ denote the solutions to \eqref[deltaNLS], \eqref[nls], respectively, with the initial data $\varphi$. Then, there exist $\eps>0$ and $C>0$ independent of $\eps$ such that $u$ and $v$ are global and they satisfy $\norm[u]_{L_t^{a}L_x^{r}(\R)} <C \norm[\varphi]_{H^1}$ and $\norm[v]_{L_t^{a}L_x^{r}(\R)} <C\norm[\varphi]_{H^1}$, if $\norm[\varphi]_{H^1}<\eps$.
\end{proposition}

For the proof of \propref[prop:SDS], see \cite[Proposition 3.2]{BV}.


\subsection{Linear Profile Decomposition and its radial version}

To prove the scattering results, we introduce the linear profile decomposition theorems. The linear profile decomposition for non-radial data, \propref[prop:LPD], is obtained in \cite{BV}.

\begin{proposition}[linear profile decomposition] \label{prop:LPD}
Let $\{ \varphi_n\}_{n\in \N}$ be a bounded sequence in $H^1(\R)$. Then, up to subsequence, we can write 
\[ \varphi_n=\sum_{j=1}^{J} e^{i t_{n}^j H_\ga} \tau_{x_n^j} \psi^j +W_n^J, \quad \forall J \in \N, \]
where 
$t_n^j\in \R$, $x_n^j \in \R$, $\psi^j \in H^1(\R)$, and the following hold. 
\begin{itemize}
\item for any fixed $j$, we have :
\begin{align*}
&\text{either } t_n^j=0 \text{ for any } n \in \N, \text{ or } t_n^j \to \pm \infty \text{ as } n\to \infty,
\\
&\text{either } x_n^j=0 \text{ for any } n \in \N, \text{ or } x_n^j \to \pm \infty \text{ as } n\to \infty.
\end{align*}
\item orthogonality of the parameters:
\[ |t_n^j -t_n^k|+|x_n^j-x_n^k| \to \infty \text{ as } n \to \infty, \quad \forall j\neq k. \]
\item smallness of the reminder:
\[ \forall \eps >0, \exists J=J(\eps) \in \N \text{ such that } \limsup_{n\to \infty} \norm[e^{-itH_\ga}W_n^J]_{L_t^\infty L_x^\infty} <\eps. \]
\item orthogonality in norms: for any $ J\in \N$
\begin{align*} 
\norm[\varphi_n]_{L^2}^2 &=\sum_{j=1}^J \norm[\psi^j]_{L^2}^2 +\norm[ W_n^J]_{L^2}^2 +o_n(1),
\\
\norm[\varphi_n]_{H}^2 &=\sum_{j=1}^J \norm[\tau_{x_n^j} \psi^j ]_{H}^2 +\norm[ W_n^J]_{H}^2 +o_n(1),
\end{align*}
where $\norm[v]_{H}^2:=\frac{1}{2} \norm[\del_x v]_{L^2}^2 -\ga |v(0)|^2$.
Moreover, we have 
\[ \norm[\varphi_n]_{L^q}^q =\sum_{j=1}^J \norm[e^{i t_n^j H_\ga} \tau_{x_n^j} \psi^j ]_{L^q}^q +\norm[ W_n^J]_{L^q}^q +o_n(1), \quad q\in (2,\infty),   \quad \forall J\in \N,\]
and in particular, for any  $J\in \N$
\begin{align*}
S_{\omega}(\varphi_n)
&=\sum_{j=1}^J  S_{\omega}(e^{i t_n^j H_\ga} \tau_{x_n^j} \psi^j ) + S_{\omega}( W_n^J) +o_n(1),
\\
I_{\omega}(\varphi_n)
&=\sum_{j=1}^J  I_{\omega}(e^{i t_n^j H_\ga} \tau_{x_n^j} \psi^j ) +I_{\omega}( W_n^J) +o_n(1). 
\end{align*}
\end{itemize}
\end{proposition}

\begin{proof}
See \cite[Theorem 2.1 and Section 2.2]{BV}.
\end{proof}

\begin{remark}
It is not clear whether 
\[P(\varphi_n)=\sum_{j=1}^J  P(e^{i t_n^j H_\ga} \tau_{x_n^j} \psi^j ) +P( W_n^J) +o_n(1), \quad \forall J\in \N \]
holds or not.  That is why we use the Nehari functional $I_{\omega}$ to prove the scattering results. 

\end{remark}

We introduce the reflection operator $\cR$ such that $\cR \varphi(x):=\varphi(-x)$. 

\propref[prop:LPD] is insufficient to prove the scattering result for radial data. 
We need the following linear profile decomposition for radial solutions, which is a key ingredient.
\begin{theorem}[linear profile decomposition for radial data]
\label{cor:LPD}
Let $\{ \varphi_n\}_{n\in \N}$ be a bounded sequence in $H_{rad}^1(\R)$. Then, up to subsequence, we can write 
\begin{align} \label{3.1}
\varphi_n
= \frac{1}{2}  \sum_{j=1}^{J} \l( e^{i t_{n}^j H_\ga} \tau_{x_n^j} \psi^j +  e^{i t_{n}^j H_\ga} \tau_{-x_n^j} \cR \psi^j \r) +\frac{1}{2} \l( W_n^J + \cR W_n^J \r),
\quad \forall J \in \N, 
\end{align}
where 
$t_n^j\in \R$, $x_n^j \in \R$, $\psi^j \in H^1(\R)$, and the following hold. 
\begin{itemize}
\item for any fixed $j$, we have :
\begin{align*}
&\text{either } t_n^j=0 \text{ for any } n \in \N, \text{ or } t_n^j \to \pm \infty \text{ as } n\to \infty,
\\
&\text{either } x_n^j=0 \text{ for any } n \in \N, \text{ or } x_n^j \to \pm \infty \text{ as } n\to \infty.
\end{align*}
\item orthogonality of the parameters:
\[ |t_n^j -t_n^k|\to \infty, \text{ or } |x_n^j-x_n^k| \to \infty \text{ and }  |x_n^j+x_n^k| \to \infty, \text{ as } n \to \infty, \quad \forall j\neq k.  \]
\item smallness of the reminder:
\[ \forall \eps >0, \exists J=J(\eps) \in \N \text{ such that } \limsup_{n\to \infty} \norm[e^{-itH_\ga}W_n^J]_{L_t^\infty L_x^\infty} <\eps. \]
\item orthogonality in norms: for any $J\in \N$,
\begin{align*} 
\norm[\varphi_n]_{L^2}^2 &= \sum_{j=1}^J \norm[ \frac{1}{2} \l(  \tau_{x_n^j} \psi^j +  \tau_{-x_n^j} \cR \psi^j \r)]_{L^2}^2 +\norm[ \frac{1}{2} \l( W_n^J + \cR W_n^J \r)]_{L^2}^2 +o_n(1),
\\
\norm[\varphi_n]_{H}^2 &=\sum_{j=1}^J \norm[\frac{1}{2} \l(  \tau_{x_n^j} \psi^j +  \tau_{-x_n^j} \cR \psi^j \r)]_{H}^2  +\norm[ \frac{1}{2} \l( W_n^J + \cR W_n^J \r)]_{H}^2 +o_n(1).
\end{align*}
Moreover, for any $q\in (2,\infty)$, we have 
\[ \norm[\varphi_n]_{L^q}^q =\sum_{j=1}^J \norm[\frac{1}{2} e^{i t_n^j H_\ga} \l(  \tau_{x_n^j} \psi^j +  \tau_{-x_n^j} \cR \psi^j \r) ]_{L^q}^q + \norm[ \frac{1}{2} \l( W_n^J + \cR W_n^J \r)]_{L^q}^q +o_n(1),  \quad \forall J\in \N,\]
and in particular, for any $J\in \N$,
\begin{align*}
S_{\omega}(\varphi_n)
&=\sum_{j=1}^J  S_{\omega}\l(\frac{1}{2} e^{i t_n^j H_\ga} \l(  \tau_{x_n^j} \psi^j +  \tau_{-x_n^j} \cR \psi^j \r) \r)+ S_{\omega}\l( \frac{1}{2} \l( W_n^J + \cR W_n^J \r)\r) +o_n(1),  
\\
I_{\omega}(\varphi_n)
&=\sum_{j=1}^J  I_{\omega}\l(\frac{1}{2} e^{i t_n^j H_\ga} \l(  \tau_{x_n^j} \psi^j +  \tau_{-x_n^j} \cR \psi^j \r) \r)+ I_{\omega}\l( \frac{1}{2} \l( W_n^J + \cR W_n^J \r)\r) +o_n(1).
\end{align*}
\end{itemize}
\end{theorem}


\begin{proof}
Since $\{\varphi_n\}$ is bounded in $H^1(\R)$, we can apply the linear profile decomposition without the radial assumption, \propref[prop:LPD],  and obtain the following: for any $J \in \N$ and $j \in \{ 1,2, \cdots, J\}$, up to subsequence, there exist $\{ t_n^j\}_{n\in \N}$, $\{ x_n^j\}_{n\in \N}$, and $\psi^j \in H^1(\R)$ such that we can write
\[ \varphi_n = \sum_{j=1}^{J} e^{it_n^j H_{\ga}} \tau_{x_n^j} \psi^j+ W_n^J. \]
Since $\varphi_n$  is radial, 
\[ 2\varphi_n(x)=\varphi_n(x)+\varphi_n(x)=\varphi_n(x)+\varphi_n(-x)=\varphi_n(x)+\cR \varphi_n(x). \]
By combining the identities, we get
\begin{align*}
2\varphi_n (x)
& = \sum_{j=1}^{J} e^{i t_{n}^j H_\ga} \tau_{x_n^j}  \psi^j +W_n^J + 
\cR \l(  \sum_{j=1}^{J} e^{i t_{n}^j H_\ga} \tau_{x_n^j}  \psi^j + W_n^J \r)
\\
& = \sum_{j=1}^{J} \l(e^{i t_{n}^j H_\ga} \tau_{x_n^j}  \psi^j + e^{i t_{n}^j H_\ga} \tau_{-x_n^j}  \cR\psi^j \r)+W_n^J + 
\cR W_n^J, 
\end{align*}
where we have used $ \cR e^{i t_{n}^j H_\ga} = e^{i t_{n}^j H_\ga} \cR$ and $\cR \tau_y = \tau_{-y} \cR$, which gives \eqref[3.1]. 

We only prove the orthogonality of the parameters. 
If $x_n^j+x_n^k \to \overline{x}\in \R$ and $t_n^j=t_n^k$ for $j<k$, then we replace $\psi^j +\tau_{-\overline{x}} \cR \psi^k$ by $\psi^j$ and $0$ by $\psi^k$ and regard the remainder terms as $W_n^J$. 
By this replacement, we have $|x_n^j-x_n^k| \to \infty$ and  $ |x_n^j+x_n^k| \to \infty$ as $n \to \infty$ when $t_n^j=t_n^k$. The orthogonality in norms follows from the orthogonality of the parameters by a standard argument. 
\end{proof}


\begin{lemma}
Let $k$ be a nonnegative integer and, for $l\in\{0,1,2,\cdots,k\}$, $\varphi_{l} \in H^1(\R)$ (or $\varphi_{l} \in H_{rad}^1(\R)$) satisfy
\begin{align*} 
\begin{array}{ll} 
S_{\omega}(\sum_{l=0}^{k} \varphi_{l})\leq m_{\omega}-\delta, 
&
S_{\omega}(\sum_{l=0}^{k} \varphi_{l})\geq \sum_{l=0}^{k}  S_{\omega}(\varphi_{l}) - \eps, 
\\
I_{\omega}(\sum_{l=0}^{k} \varphi_{l})\geq -\eps, 
&
I_{\omega}(\sum_{l=0}^{k} \varphi_{l})\leq \sum_{l=0}^{k} I_{\omega}(\varphi_{l}) + \eps,
\end{array}
\end{align*}
for $\delta$, $\eps$ satisfying $2\eps<\delta$. Then $\varphi_{l} \in \cM_{\omega}^{+}$ for all $l\in\{0,1,2,\cdots,k\}$. Namely, we have $0\leq S_{\omega}(\varphi_{l}) < m_{\omega}$ and $I_{\omega}(\varphi_{l})\geq 0$ for all $l\in\{0,1,2,\cdots,k\}$.
\end{lemma}

\begin{proof}
We assume that there exists an $l \in \{ 0,1,2,\cdots, k\}$ such that $I_{\omega}(\varphi_{l})<0$. By the definition of $m_{\omega}$ and the positivity of $J_{\omega}=J_{\omega}^{1,0}$, we obtain
\begin{align*}
m_{\omega}
& \leq \sum_{l=0}^{k} J_{\omega}(\varphi_{l})
\\
& = \sum_{l=0}^{k} \l(S_{\omega}(\varphi_{l}) - \frac{1}{2}I_{\omega}(\varphi_{l})\r)
\\
& = \sum_{l=0}^{k} S_{\omega}(\varphi_{l}) -  \frac{1}{2} \sum_{l=0}^{k}  I_{\omega}(\varphi_{l})
\\
& \leq S_{\omega}\l(\sum_{l=0}^{k} \varphi_{l}\r) +\eps - \frac{1}{2} \l(I_{\omega}\l(\sum_{l=0}^{k} \varphi_{l}\r) - \eps \r)
\\
& \leq m_{\omega}-\delta +\eps  +\eps 
\\
& < m_{\omega}.
\end{align*}
This is a contradiction. So, $I_{\omega}(\varphi_{l})\geq 0$ for all $l\in\{0,1,2,\cdots,k\}$. Moreover, for any $l\in\{0,1,2,\cdots,k\}$, we have
\begin{align*}
S_{\omega}(\varphi_{l}) = J_{\omega}(\varphi_{l}) + \frac{1}{2} I_{\omega}(\varphi_{l}) \geq 0,
\end{align*}
and 
\begin{align*}
S_{\omega}(\varphi_{l}) \leq \sum_{l=0}^{k}  S_{\omega}( \varphi_{l})  \leq S_{\omega}\l( \sum_{l=0}^{k} \varphi_{l}\r) + \eps \leq m_{\omega} -\delta +\eps <m_{\omega}.
\end{align*}
Therefore, we get $\varphi_{l} \in \cM_{\omega}^{+}$ for all $l\in\{0,1,2,\cdots,k\}$.
\end{proof}


\subsection{Perturbation Lemma and Nonlinear Profile Decomposition}
We use a perturbation lemma and lemmas for nonlinear profiles. The proofs of these results are same as in the defocusing case (see \cite{BV}). 

\begin{lemma} \label{lem:perturb}
For any $M>0$, there exist $\eps=\eps(M)>0$ and $C=C(M)>0$ such that the following occurs. 
Let $v \in C(\R:H^1(\R))\cap L_t^{a}(\R:L_x^r(\R))$ be a solution of the integral equation with source term $e$:
\[ v(t)=e^{-itH_{\ga}}\varphi +i\int_{0}^{t} e^{-i(t-s)H_\ga} (|v(s)|^{p-1}v(s))ds +e (t)\]
with $\norm[v]_{L_t^a L_x^r}<M$ and $\norm[e]_{L_t^{a}L_x^r}<\eps$. Assume moreover that $\varphi_0 \in H^1(\R)$ is such that $\norm[e^{-itH_{\ga}}\varphi_0]_{L_t^{a}L_x^r} < \eps$, then the solution $u(t,x)$ to \eqref[deltaNLS] with initial condition $\varphi+\varphi_0$:
\[ u(t)=e^{-itH_{\ga}}(\varphi+\varphi_0) +i \int_{0}^{t} e^{-i(t-s)H_{\ga}}(|u(s)|^{p-1}u(s)) ds,\]
satisfies $u \in L_t^{a}L_x^r$ and moreover $\norm[u-v]_{L_t^{a}L_x^r}<C\eps$. 
\end{lemma}

See \cite[Proposition 4.7]{FXC} and \cite[Proposition 3.3]{BV} for the proof. 

Following Lemma \ref{lemma3.9}, \ref{lemma3.10}, and \ref{lemma3.11} can be proved in the same manner as \cite[Proposition 3.4, 3.5, 3.6]{BV}, respectively.

\begin{lemma} \label{lemma3.9}
Let $\{ x_n\}_{n\in \N}$ be a sequence of real numbers such that $|x_n| \to \infty$ as $n \to \infty$, $u_0 \in H^1(\R)$ and $U \in C(\R:H^1(\R))\cap L_t^{a}(\R:L_x^r(\R))$ be a solution of \eqref[nls] with the initial data $u_0$. Then we have
\[ U_n(t)=e^{-itH_{\ga}} \tau _{x_n} u_{0}+i \int_{0}^{t} e^{-i(t-s)H_{\gamma}} (|U_n(s)|^{p-1}U_n(s)) ds +g_n(t),\]
where $U_n(t,x)=U(t,x-x_n)$ and $\norm[g_n]_{L_t^{a}L_x^r}\to 0 $ as $n\to\infty$.
\end{lemma}

\begin{lemma} \label{lemma3.10}
Let $\varphi \in H^1(\R)$. Then there exist solutions $W_{\pm} \in C(\R_{\pm}:H^1(\R))\cap L_t^{a}(\R_{\pm}:L_x^r(\R))$ to \eqref[deltaNLS] such that 
\[ \norm[W_{\pm} (t,\cdot)-e^{-itH_{\ga}} \varphi]_{H^1} \to 0, \text{ as } t \to \pm \infty. \]
Moreover, if $\{t_n\}_{n\in \N}$ is such that $t_n \to \mp \infty$ as $n\to \infty$ and $W_{\pm}$ is global, then 
\[ W_{\pm,n}(t)=e^{-itH_{\ga}}  \varphi_{n}+i \int_{0}^{t} e^{-i(t-s)H_{\gamma}} (|W_{\pm,n}(s)|^{p-1}W_{\pm,n}(s)) ds +f_{\pm,n}(t),\]
where $\varphi_n=e^{i t_n H_{\ga}}\varphi$, $W_{\pm,n}(t,x)=W_{\pm}(t-t_n, x)$, $\norm[f_{\pm,n}]_{L_t^{a}L_x^r}\to 0$ as $n\to \infty$, and double-sign corresponds.
\end{lemma}

\begin{lemma} \label{lemma3.11}
Let $\{t_n\}_{n \in \N}$, $\{ x_n\}_{n\in \N}$ be sequences of real numbers such that $t_n \to \mp \infty$ and $|x_n| \to \infty$ as $n \to \infty$, $\varphi \in H^1(\R)$ and $V_{\pm} \in C(\R_{\pm}:H^1(\R))\cap L_t^{a}(\R_{\pm}:L_x^r(\R))$ be solutions of \eqref[nls] such that 
\[ \norm[V_{\pm} (t,\cdot)-e^{-itH_{0}} \varphi]_{H^1} \to 0 \text{ as } t \to \pm \infty.\]
 Then we have
\[ V_{\pm,n}(t,x)=e^{-itH_{\ga}} \varphi_{n}+i \int_{0}^{t} e^{-i(t-s)H_{\gamma}} (|V_{\pm,n}(s)|^{p-1}V_{\pm,n}(s)) ds +e_{\pm,n}(t,x),\]
where $\varphi_{n}=e^{i t_n H_{\ga}}\tau_{x_n} \varphi$, $V_{\pm,n}(t,x)=V_{\pm}(t-t_n,x-x_n)$, $\norm[e_{\pm,n}]_{L_t^{a}L_x^r}\to 0 $ as $n\to\infty$, and double-sign corresponds.
\end{lemma}


\subsection{Construction of a Critical Element}

We define the critical action level $S_{\omega}^c$ for fixed $\omega$ as follows.
\[ S_\omega^c:= \sup \{ S: S_{\omega}(\varphi) <S \text{ for any } \varphi \in \cM_{\omega}^{+} \text{ implies } u\in L_t^{a}L_x^r \}.\]
By the small data scattering result \propref[prop:SDS], we obtain $S_{\omega}^c>0$. We prove $S_{\omega}^c =m_\omega$ by the contradiction argument.  

We assume $S_{\omega}^c< m_\omega$. 
By this assumption, we can take a sequence $\{\varphi_n \}_{n \in \N}\subset \cM_{\omega}^{+}$ such that $S_{\omega}(\varphi_n) \to S_{\omega}^{c}$ as $n \to \infty$, and $\norm[u_n]_{L_t^{a}L_x^r(\R)} = \infty$ for all $n \in \N$, where $u_n$ is a global solution to \eqref[deltaNLS] with the initial data $\varphi_n$. Then, we obtain the following lemma.

\begin{lemma}[critical element]
We assume that  $S_{\omega}^c< m_\omega$. Then we find a  global solution $u^c\in C(\R:H^1(\R))$  of \eqref[deltaNLS] which satisfies $u^c(t) \in \cM_{\omega}^{+}$ for any $t\in \R$ and 
\begin{align*}
S_{\omega}(u^c)=S_{\omega}^{c},
\quad
\norm[u^c]_{L_t^{a}L_x^r(\R)}=\infty.
\end{align*}
\end{lemma}
This $u^c$ is called a critical element.

\begin{proof}
First, we consider the non-radial case. 

\noindent{\bf Case1: non-radial data.} 
By $\varphi_{n} \in \cN_{\omega}^{+}$ and Lemma 2.16, we have $\norm[\varphi_{n}]_{H^1}^2 \cleq \norm[\varphi_{n}]_{\cH}^2 \cleq E(\varphi_n)+\omega M(\varphi_n)<n_{\omega}$ for all $n \in \N$. 
Since $\{\varphi_n\}$ is a bounded sequence in $H^1(\R)$,
we can apply the linear profile decomposition, \propref[prop:LPD], to decompose $\varphi_n$. 
\[ \varphi_n=\sum_{j=1}^{J} e^{i t_n^j H_\ga} \tau_{x_n^j} \psi^j +W_{n}^{J}, \quad \forall J \in \N. \]
By the orthogonality of the functionals in \propref[prop:LPD], we have
\begin{align*} 
S_{\omega}(\varphi_n) &=\sum_{j=1}^{J} S_{\omega} (e^{i t_n^j H_\ga}  \tau_{x_n^j} \psi^j ) +S_{\omega} (W_n^J) +o_n(1),
\\
I_{\omega}(\varphi_n) &=\sum_{j=1}^{J} I_{\omega} (e^{i t_n^j H_\ga}  \tau_{x_n^j} \psi^j ) +I_{\omega} (W_n^J) +o_n(1),
\end{align*}
wehre $o_n(1)\to 0$ as $n \to \infty$. 

By these decompositions and $S_{\omega}(\varphi_n)<n_{\omega}$, we can find  $\delta,\eps>0$ satisfying $2\eps<\delta$ and
\begin{align*} 
& S_{\omega}(\varphi_n )\leq n_{\omega}-\delta, 
\\
& S_{\omega}(\varphi_n)\geq \sum_{j=0}^{J}  S_{\omega}(e^{i t_{n}^j H_\ga} \tau_{x_n^j} \psi^j)+S_{\omega}(W_n^J) - \eps, 
\\
& I_{\omega}(\varphi_{n})\geq -\eps, 
\\
& I_{\omega}( \varphi_{n})\leq \sum_{j=0}^{J}  I_{\omega}(e^{i t_{n}^j H_\ga} \tau_{x_n^j} \psi^j)+I_{\omega}(W_n^J) + \eps,
\end{align*}
for large $n$. Therefore, by Lemma 3.6, we see that
\[e^{i t_{n}^j H_\ga} \tau_{x_n^j} \psi^j \in \cN_{\omega}^{+} \text{ and } W_n^J \in \cN_{\omega}^{+} \text{ for large }n,\] 
which means that 
\[S_{\omega}(e^{i t_{n}^j H_\ga} \tau_{x_n^j} \psi^j )\geq 0\text{ and }S_{\omega} (W_n^J) \geq 0 \text{ for large }n.\] So, we have 
\[ S_{\omega}^{c} = \limsup_{n \to \infty} S_{\omega}(\varphi_n) \geq \limsup_{n \to \infty} \sum_{j=1}^{J} S_{\omega} (e^{i t_n^j H_\ga}  \tau_{x_n^j} \psi^j ), \]
for any $J$. 
We prove $S_{\omega}^c = \limsup_{n \to \infty} S_{\omega} (e^{i t_n^j H_\ga}  \tau_{x_n^j} \psi^j)$ for some $j$. We may $j=1$ by reordering. If this is proved, then we find that $J=1$ and $W_n^J \to 0$ in $L_t^\infty H_x^1$ as $n \to \infty$. Indeed, $\limsup_{n\to \infty} S_{\omega}(W_n^1)=0$ holds and thus $\limsup_{n\to \infty}  \norm[W_n^1]_{H^1}  =0$ holds by $\norm[W_n^1]_{H^1} \ceq S_{\omega}(W_n^1)$ since $W_n^1$ belongs to $\cN_{\omega}^{+}$ for large $n \in \N$. 
On the contrary, we assume that $S_{\omega}^c = \limsup_{n \to \infty} S_{\omega} (e^{-i t_n^j H_\ga}  \tau_{x_n^j} \psi^j)$ fails for all $j$. Then, for all $j$, there exists $\delta=\delta_j>0$ such that $\limsup_{n \to \infty} S_{\omega} (e^{i t_n^j H_\ga}  \tau_{x_n^j} \psi^j) < S_{\omega}^c -\delta$. 
By reordering, we can choose $0\leq J_{1}\leq J_{2} \leq J_{3} \leq J_{4} \leq J_5 \leq J$ such that
\begin{align*}
\begin{array}{cccc}
1\leq j \leq J_1: & \qquad t_n^j=0, \ \forall n & \text{ \quad and \quad }  &x_n^j=0, \ \forall n,
\\
J_1+1 \leq j \leq J_2: & \qquad t_n^j=0, \ \forall n & \text{ \quad and \quad } & \lim_{n\to \infty} |x_n^j|=\infty,
\\
J_2 +1\leq j \leq J_3: & \qquad \lim_{n\to\infty }t_n^j=+\infty ,&  \text{ \quad and \quad }  & x_n^j=0, \ \forall n,
\\
J_3+1 \leq j \leq J_4: & \qquad \lim_{n\to\infty }t_n^j=-\infty , &  \text{ \quad and \quad } & x_n^j=0, \ \forall n,
\\
J_4 +1\leq j \leq J_5: & \qquad \lim_{n\to\infty }t_n^j=+\infty , &  \text{ \quad and \quad }  &\lim_{n\to \infty} |x_n^j|=\infty, 
\\
J_5+1 \leq j \leq J: & \qquad \lim_{n\to\infty }t_n^j=-\infty , & \text{ \quad and \quad } & \lim_{n\to \infty} |x_n^j|=\infty.
\end{array}
\end{align*} 
Above we are assuming that if $a>b$, then there is no $j$ such that $a\leq j \leq b$. Notice that $J_1 \in \{0,1\}$ by the orthogonality of the parameters. 
We may treat only the case $J_1=1$ here. The case $J_1=0$ is easier.  We have $0<S_{\omega}(\psi^1) <S_{\omega}^c-\delta$ by $(t_n^j,x_n^j)=(0,0)$ and the assumption. Hence, by the definition of $S_{\omega}^c$, we can find $N \in C(\R:H^1(\R)) \cap L_t^{a}(\R: L_x^r(\R))$ such that 
\[ N(t,x)= e^{-it H_\ga } \psi^1 +i \int_{0}^{t} e^{-i(t-s) H_\ga} (|N(s)|^{p-1}N(s))ds. \]

For every $j$ such that $J_1+1 \leq j \leq J_2$, 
let $U^j$ be the solution of \eqref[nls] with the initial data $\psi^j$.
Since we have $\tau_{x_n^j} \psi^j \in \cN_{\omega}^{+}$, $\psi^j$ satisfies that $S_{\omega, 0}(\psi^j) \leq S_{\omega}(\tau_{x_n^j} \psi^j) \leq S_{\omega}^c <n_{\omega}= l_{\omega}$ and $P_0(\psi^j) \geq 0$. (since $0>P_{0}(\psi^j)=\lim_{n\to \infty} P(\tau_{x_n^j}\psi^j) \geq 0$ if we assume $P_{0}(\psi^j)<0$.) Therefore, we see that the solution $U^j$ scatters by \cite{FXC} and \cite{AN}, that is, $U^j (t,x) \in C(\R:H^1(\R)) \cap L_t^{a}(\R: L_x^r(\R))$. We set $U_n^j(t,x):=U^j(t,x-x_n^j)$.

For every $j$ such that $J_2+1 \leq j \leq J_3$, we associate with profile $\psi^j$ the function $W_{-}^j (t,x) \in C(\R_{-}:H^1(\R)) \cap L_t^{a}(\R_{-}: L_x^r(\R))$  by \lemref[lemma3.10]. We claim that $W_{-}^j (t,x) \in C(\R:H^1(\R)) \cap L_t^{a}(\R: L_x^r(\R))$.  Indeed,  by the assumption, we see that $S_{\omega}(W_{-}^j)=\lim_{n\to \infty} S_{\omega}(e^{it_n^jH_{\ga} }\psi^j) <S_{\omega}^c$, since $e^{it_n^j H_{\ga}} \psi^j \to W_{-}^j$ in $H^1(\R)$ with $t_n^j \to \infty$ as $n \to \infty$. . Therefore,  by the definition of $S_{\omega}^c$, we obtain $W_{-}^j (t,x) \in C(\R:H^1(\R)) \cap L_t^{a}(\R: L_x^r(\R))$. We set $W_{-,n}^j (t,x):=W_{-}^j (t-t_n^j,x)$.

For every $j$ such that $J_3+1 \leq j \leq J_4$, we associate with profile $\psi^j$ the function $W_{+}^j (t,x) \in C(\R_{+}:H^1(\R)) \cap L_t^{a}(\R_{+}: L_x^r(\R))$ by \lemref[lemma3.10].  And the same argument as above gives us that $W_{+}^j (t,x) \in C(\R:H^1(\R)) \cap L_t^{a}(\R: L_x^r(\R))$.  We set $W_{+,n}^j (t,x):=W_{+}^j (t-t_n^j,x)$.

For every $j$ such that $J_4+1 \leq j \leq J_5$, we associate with profile $\psi^j$ the function $V_{-}^j (t,x) \in C(\R_{-}:H^1(\R)) \cap L_t^{a}(\R_{-}: L_x^r(\R))$  by \lemref[lemma3.11]. We prove $V_{-}^j (t,x) \in C(\R:H^1(\R)) \cap L_t^{a}(\R: L_x^r(\R))$. Now, $\limsup_{n\to \infty} S_{\omega}(e^{it_n^j H_{\ga}}\tau_{x_n^j} \psi^j)< S_{\omega}^c-\delta$ holds by the assumption.  Here, since $e^{-itH_{\ga}}$ is unitary in $L^2(\R)$ and conserves the linear energy, and $\gamma \leq 0$, we have
\begin{align*}
S_{\omega}(e^{it_n^j H_{\ga}}\tau_{x_n^j} \psi^j)
& = E(e^{it_n^j H_{\ga}}\tau_{x_n^j} \psi^j) + \omega M(e^{it_n^j H_{\ga}}\tau_{x_n^j} \psi^j)
\\
& = \norm[\tau_{x_n^j} \psi^j]_{\cH}^2 - \frac{1}{p+1}\norm[e^{it_n^j H_{\ga}}\tau_{x_n^j} \psi^j]_{L^{p+1}}^{p+1} 
\\
& \geq \frac{1}{4} \norm[\del_x (\tau_{x_n^j} \psi^j)]_{L^2}^2 +\frac{\omega}{2} \norm[\tau_{x_n^j} \psi^j]_{L^2}^2 - \frac{1}{p+1}\norm[e^{it_n^j H_{\ga}}\tau_{x_n^j} \psi^j]_{L^{p+1}}^{p+1} 
\\
& =\frac{1}{4} \norm[\del_{x} \psi^j]_{L^2}^2 +\frac{\omega}{2} \norm[\psi^j]_{L^2}^2 - \frac{1}{p+1}\norm[e^{it_n^j H_{\ga}}\tau_{x_n^j} \psi^j]_{L^{p+1}}^{p+1}. 
\end{align*}
Since $t_n^j \to \infty$, we have $\norm[e^{it_n^j H_{\ga}}\tau_{x_n^j} \psi^j]_{L^{p+1}}^{p+1} \to 0$ as $n \to \infty$ by \cite[Section 2, (2.4)]{BV}. Therefore, we obtain $\frac{1}{4} \norm[\del_{x} \psi^j]_{L^2}^2 +\frac{\omega}{2} \norm[\psi^j]_{L^2}^2 \leq S_{\omega}^c - \delta$. Since $\psi^j$ is the final state of $V_{-}^j$, we have $S_{\omega,0}(V_{-}^j)=\frac{1}{4} \norm[\del_{x} \psi^j]_{L^2}^2 +\frac{\omega}{2} \norm[\psi^j]_{L^2}^2 \leq S_{\omega}^c - \delta <n_{\omega}=l_{\omega}$. By \cite{FXC} and \cite{AN}, 
we have $V_{-}^j (t,x) \in C(\R:H^1(\R)) \cap L_t^{a}(\R: L_x^r(\R))$. We set $V_{-,n}^j (t,x):=V_{-}^j (t-t_n^j,x-x_n^j)$.

For every $j$ such that $J_5+1 \leq j \leq J$, we associate with profile $\psi^j$ the function $V_{+}^j (t,x) \in C(\R_{+}:H^1(\R)) \cap L_t^{a}(\R_{+}: L_x^r(\R))$ by \lemref[lemma3.11]. And the same argument as above gives us that $V_{+}^j (t,x) \in C(\R:H^1(\R)) \cap L_t^{a}(\R: L_x^r(\R))$.  We set $V_{+,n}^j (t,x):=V_{+}^j (t-t_n^j,x-x_n^j)$.

We define the nonlinear profile as follows.
\[
Z_n^J := N+ \sum_{j=J_1+1}^{J_2} U_n^j + \sum_{j=J_2+1}^{J_3} W_{-,n}^j + \sum_{j=J_3+1}^{J_4} W_{+,n}^j
+ \sum_{j=J_4+1}^{J_5} V_{-,n}^j + \sum_{j=J_5+1}^{J_6} V_{+,n}^j.
\]
By Lemma \ref{lemma3.9}, \ref{lemma3.10}, \ref{lemma3.11}, we have 
\[ Z_n^J = e^{-it H_\ga} (\varphi_n - W_n^J) +iz_n^J +r_n^J, \]
where $\norm[r_n^J]_{L_t^a L_x^r} \to 0$ as $n \to \infty$ and 
\begin{align*}
&z_n^J (t) 
:= \int_{0}^{t} e^{-i(t-s) H_\ga} (|N(s)|^{p-1}N(s))ds 
+ \sum_{j=J_1+1}^{J_2} \int_{0}^{t} e^{-i(t-s) H_\ga} (|U_n^j(s)|^{p-1}U_n^j(s))ds 
\\
&+ \sum_{j=J_2+1}^{J_3} \int_{0}^{t} e^{-i(t-s) H_\ga} (|W_{-,n}^j(s)|^{p-1}W_{-,n}^j(s))ds 
+ \sum_{j=J_3+1}^{J_4} \int_{0}^{t} e^{-i(t-s) H_\ga} (|W_{+,n}^j(s)|^{p-1}W_{+,n}^j(s))ds 
\\
&+ \sum_{j=J_4+1}^{J_5} \int_{0}^{t} e^{-i(t-s) H_\ga} (|V_{-,n}^j(s)|^{p-1}V_{-,n}^j(s))ds  
+ \sum_{j=J_5+1}^{J} \int_{0}^{t} e^{-i(t-s) H_\ga} (|V_{+,n}^j(s)|^{p-1}V_{+,n}^j(s))ds.
\end{align*}
We also have 
\[ \norm[z_n^J -\int_{o}^{t} e^{-i(t-s)H_\gamma} (|Z_n^J(s)|^{p-1} Z_n^J(s)) ds]_{L_t^a L_x^r} \to 0, \text{ as } n \to \infty. \]
Therefore, we get
\[  Z_n^J = e^{-itH_\ga }(\varphi_n - W_n^J) +i \int_{o}^{t} e^{-i(t-s)H_\gamma} (|Z_n^J(s)|^{p-1} Z_n^J(s)) ds + s_n^J, \]
with $\norm[s_n^J]_{L_t^a L_x^r} \to 0 $ as $n \to \infty$. In order to apply the perturbation lemma, \lemref[lem:perturb], we need a bound on $\sup_{J} (\limsup_{n\to \infty} \norm[Z_n^J]_{L_t^a L_x^r})$. We have 
\begin{align*}
\limsup_{n\to \infty} (\norm[Z_n^J]_{L_t^a L_x^r})^{p}
&\leq 2\norm[N]_{L_t^a L_x^r}^{p}
+2\sum_{j=J_1+1}^{J_2} \norm[U^j]_{L_t^a L_x^r}^{p}
\\
&+2\sum_{j=J_2+1}^{J_3} \norm[W_{-}^j]_{L_t^a L_x^r}^{p}
+2\sum_{j=J_3+1}^{J_4} \norm[W_{+}^j]_{L_t^a L_x^r}^{p}
\\
&+2\sum_{j=J_4+1}^{J_5} \norm[V_{-}^j]_{L_t^a L_x^r}^{p}
+2\sum_{j=J_5+1}^{J} \norm[V_{+}^j]_{L_t^a L_x^r}^{p},
\end{align*}
where we have used Corollary A.2 in \cite{BV}. 
For simplicity, $a_n^j$ denotes $2\norm[N]_{L_t^a L_x^r}^{p}$ if $1 \leq j \leq J_1$, $2\norm[U_n^j]_{L_t^a L_x^r}^{p}=2\norm[U^j]_{L_t^a L_x^r}^{p}$ if $J_1 + 1 \leq j  \leq J_2$, and so on. Then, the above inequality means $\limsup_{n\to \infty} (\norm[Z_n^J]_{L_t^a L_x^r})^{p} \leq \sum_{j=1}^{J} a_n^j$.

There exists a finite set $\cJ$ such that $\norm[\psi^j]_{H^1}< \eps_0$ for any $j \not\in \cJ$, where $\eps_0$ is the universal constant in the small data scattering result, Proposition \ref{prop:SDS}. By Proposition \ref{prop:SDS} and the orthogonalities in $H$-norm and $L^2$-norm, 
\begin{align*}
\limsup_{n\to \infty} (\norm[Z_n^J]_{L_t^a L_x^r})^{p} 
& \leq \limsup_{n\to \infty}  \sum_{j=1}^{J} a_n^j
 = \limsup_{n\to \infty}  \sum_{j\in \cJ} a_n^j + \limsup_{n\to \infty}  \sum_{j \not\in \cJ}  a_n^j 
\\
& \cleq \limsup_{n\to \infty}  \sum_{j\in \cJ}  a_n^j + \limsup_{n\to \infty}  \sum_{j \not\in \cJ}  \norm[e^{it_n^j H_{\ga}} \tau_{x_n^j} \psi^j ]_{\cH}
\\
& \cleq \limsup_{n\to \infty}  \sum_{j\in \cJ}  a_n^j + \limsup_{n\to \infty}  \norm[ \varphi_n]_{\cH}
\\
& \cleq \limsup_{n\to \infty}  \sum_{j\in \cJ}  a_n^j + n_{\omega}
\\
&  \cleq  \sum_{j\in \cJ}  a^j + n_{\omega}
\leq M,
\end{align*}
where $M$ is independent of $J$.

By \lemref[lem:perturb] and \propref[prop:LPD], we can choose $J$ large enough in such a way that $\limsup_{n\to \infty} \norm[e^{-itH_\gamma} W_n^J]_{L_t^a L_x^r} < \eps$, where $\eps =\eps(M)>0$. Then, we get the fact that $u_n$ scatters for large $n$, and this contradicts $\norm[u_n]_{L_t^a L_x^r}=\infty$. 

Therefore, we obtain $J=1$ and 
\[ \varphi_n = e^{it_n^1 H_\gamma} \tau_{x_n^1} \psi^1 +W_n^1, \quad S_{\omega}^c = \limsup_{n\to \infty}S_{\omega}(e^{it_n^1 H_\gamma} \tau_{x_n^1} \psi^1), \quad W_n^1 \to 0 \text{ in } L_t^\infty H_x^1.\]
By the same argument as \cite{BV}, 
we get $x_n^1=0$. 
Let $u^c$ be the nonlinear profile associated with $\psi^1$. 
Then, $S_{\omega}^c = S_{\omega}( u^s)$ and the global solution $u^c$ does not scatter by a contradiction argument and the perturbation lemma (see the proof of Proposition 6.1 in \cite{FXC} for more detail). 

\noindent{\bf Case2: radial data.} 
We only focus on the difference of the proof between the radial case and the non-radial case. 
This is in the profiles. 
By the linear profile decomposition for the radial data \thmref[cor:LPD], we have
\[ \varphi_n
= \frac{1}{2}  \sum_{j=1}^{J} \l( e^{i t_{n}^j H_\ga} \tau_{x_n^j} \psi^j +  e^{i t_{n}^j H_\ga} \tau_{-x_n^j} \cR \psi^j \r) +\frac{1}{2} \l( W_n^J + \cR W_n^J \r),
\quad \forall J \in \N. \]

For every $j$ such that $J_1+1 \leq j \leq J_2$,
let $U^j$ be the solution to \eqref[nls] with the initial data $\psi^j/2$.
Since we have $\tau_{x_n^j} \psi^j/2+\tau_{-x_n^j}\cR \psi^j/2 \in \cR_{\omega}^{+}$, $\psi^j$ satisfies that $S_{\omega, 0}(\psi^j/2) < l_{\omega}$ and $P_{0}(\psi^j/2) \geq 0$. Indeed, if we assume $S_{\omega, 0}(\psi^j/2)\geq l_{\omega}$,  then by \thmref[cor:LPD] and $\gamma \leq 0$,
\begin{align*} 
r_{\omega} 
& > S_{\omega}^c \geq \limsup_{n \to \infty} S_{\omega}(\varphi_n) \geq \limsup_{n \to \infty} ( S_{\omega}(\tau_{x_n^j}  \psi^j/2) +S_{\omega}(\tau_{-x_n^j}  \cR\psi^j/2) )
\\
& \geq \limsup_{n \to \infty}( S_{\omega,0}(\tau_{x_n^j}  \psi^j/2) +S_{\omega,0}(\tau_{-x_n^j}  \cR\psi^j/2) )
= S_{\omega,0}(  \psi^j/2) +S_{\omega,0}(\psi^j/2) \geq 2l_{\omega}.
\end{align*}
This contradicts $r_{\omega}\leq 2l_{\omega}$. 
Moreover, we see that $2 P_0(\psi^j/2)=\limsup_{n\to \infty}( P_0(\tau_{x_n^j}\psi^j/2)+ P_0(\tau_{-x_n^j} \cR\psi^j/2))= \limsup_{n\to\infty} P(\tau_{x_n^j}\psi^j/2 + \tau_{-x_n^j} \cR\psi^j/2) \geq 0$. Therefore, 
by \cite{FXC} and \cite{AN}, we have $U^j (t,x) \in C(\R:H^1(\R)) \cap L_t^{a}(\R: L_x^r(\R))$. We set $U_n^j(t,x):=U^j(t,x-x_n^j)$.

For every $j$ such that $J_4+1 \leq j \leq J_5$, we associate with profile $\psi^j$ the function $V_{-}^j (t,x) \in C(\R_{-}:H^1(\R)) \cap L_t^{a}(\R_{-}: L_x^r(\R))$  by \lemref[lemma3.11]. We prove $V_{-}^j (t,x) \in C(\R:H^1(\R)) \cap L_t^{a}(\R: L_x^r(\R))$. Now, by the assumption, we have \[ \limsup_{n\to \infty} 2 S_{\omega}(e^{it_n^j H_{\ga}}\tau_{x_n^j} \psi^j /2) =\limsup_{n\to \infty} \{  S_{\omega}(e^{it_n^j H_{\ga}}\tau_{x_n^j} \psi^j/2) +S_{\omega}(\cR e^{it_n^j H_{\ga}}\tau_{x_n^j} \psi^j/2)  \}< S_{\omega}^c-\delta. \] 
In the same argument as that for $V_{-}^j$ in the non-radial case,  we obtain $\frac{1}{4} \norm[\del_{x} \psi^j/2]_{L^2}^2 +\frac{\omega}{2} \norm[\psi^j/2]_{L^2}^2 \leq (S_{\omega}^c - \delta )/2$. Now, since $\psi^j$ is the final state of $V_{-}^j$, we have $S_{\omega,0}(V_{-}^j)=\frac{1}{4} \norm[\del_{x} \psi^j/2]_{L^2}^2 +\frac{\omega}{2} \norm[\psi^j/2]_{L^2}^2 \leq (S_{\omega}^c - \delta)/2 <r_{\omega}/2 \leq l_{\omega}$. By \cite{FXC} and \cite{AN}, we have  $V_{-}^j (t,x) \in C(\R:H^1(\R)) \cap L_t^{a}(\R: L_x^r(\R))$. We set $V_{-,n}^j (t,x):=V_{-}^j (t-t_n^j,x-x_n^j)$.

Other statements are same as in the non-radial case. This completes the proof. 
\end{proof}


\subsection{Extinction of the Critical Element}

We assume that $ \norm[u^c]_{L_t^{a}((0,\infty): L_x^r)}=\infty$, where such $u^c$ is called a forward critical element, and we prove $u^c=0$. 
In the case of $ \norm[u^c]_{L_t^{a}((-\infty,0): L_x^r)}=\infty$, the same argument as below does work. 

\begin{lemma}
\label{lem3.12}
Let $u$ be a forward critical element. Then the orbit of $u$, $\{ u(t,x): t>0 \}$, is precompact in $H^1(\R)$. And then, for any $\eps>0$, there exists $R>0$ such that 
\[  \int_{|x|>R}  |\del_x u(t,x)|^2 dx+\int_{|x|>R}  |u(t,x)|^2 dx+\int_{|x|>R}  |u(t,x)|^{p+1} dx <\eps, \text{ for any } t\in \R_+.  \]
\end{lemma}


This lemma is obtained in the same way as the defocusing case (see \cite{BV}). 

Now, we prove $u=0$ by the localized virial identity and contradiction argument. Let $u\neq 0$. For $\phi:\R_+ \to \R$, we define a function $I$ by
\[ I(t):=\int_{\R} \phi(|x|) |u(t,x)|^2 dx. \] 
Then, by a direct calculation and using \eqref[deltaNLS],  we have
\begin{align*}
I'(t)
&= \im \int_{\R} \del_x(\phi(|x|)) \overline{u(t,x)}\del_x u(t,x)dx,
\\
I''(t) 
&=\re \int_{\R} \del_x^2(\phi(|x|)) |\del_x u(t,x)|^2 dx
-\gamma \del_x^2(\phi(|x|))|_{x=0} |u(t,0)|^2
\\
& \qquad- \frac{p-1}{p+1} \re \int_{\R} \del_x^2(\phi(|x|))  |u(t,x)|^{p+1} dx 
- \frac{1}{4} \re \int_{\R} \del_x^4 (\phi(|x|)) |u(t,x)|^2 dx
\\ & \qquad
-2\gamma \re \{ \del_x (\phi (|x|))|_{x=0} u(t,0) \overline{\del_x u(t,0)} \}.
\end{align*}
Taking $\phi=\phi(r)$ such that, for $R>0$,
\[ 0\leq \phi \leq r^2, \quad | \phi' | \cleq r, \quad |\phi''| \leq 2, \quad |\phi^{(4)}| \leq \frac{4}{R^2}, \]
and 
\[ \phi (r)=\l\{ \begin{array}{ll} r^2, & 0\leq r  \leq R, \\ 0,& r\geq 2R, \end{array} \r.\]
we obtain
\begin{align} \label{eq3.3}
I''(t) 
&=4P(u(t))+
\re \int_{\R} \{ \del_x^2(\phi(|x|))-2\} |\del_x u(t,x)|^2 dx
\\ \notag & \quad
- \frac{p-1}{p+1} \re \int_{\R} \{ \del_x^2(\phi(|x|))-2\}  |u(t,x)|^{p+1} dx 
- \frac{1}{4} \re \int_{\R} \del_x^4 (\phi(|x|)) |u(t,x)|^2 dx
\\ \notag
&= 4P(u(t))+R_1+R_2+R_3,
\end{align}
where
\begin{align*}
R_1&:=\re \int_{\R} \{ \del_x^2(\phi(|x|))-2\} |\del_x u(t,x)|^2 dx,
\\
R_2&:=- \frac{p-1}{p+1} \re \int_{\R} \{ \del_x^2(\phi(|x|))-2\}  |u(t,x)|^{p+1} dx ,
\\
R_3&:=- \frac{1}{4} \re \int_{\R} \del_x^4 (\phi(|x|)) |u(t,x)|^2 dx.
\end{align*}
 By the property of $\phi$, we have
\begin{align*}
|R_1|&=\l|\re \int_{\R} \{ \del_x^2(\phi(|x|))-2\} |\del_x u(t,x)|^2 dx \r|
\leq C \int_{|x|>R}  |\del_x u(t,x)|^2 dx,
\\
|R_2|&=\l|\frac{p-1}{p+1} \re \int_{\R} \{ \del_x^2(\phi(|x|))-2\}  |u(t,x)|^{p+1} dx \r|
\leq C \int_{|x|>R}  |u(t,x)|^{p+1} dx, 
\\
|R_3|&=\l|\frac{1}{4} \re \int_{\R} \del_x^4 (\phi(|x|)) |u(t,x)|^2 dx\r|
\leq C \int_{|x|>R}  |u(t,x)|^2 dx.
\end{align*}
Therefore, we obtain
\begin{align*}
I''(t) 
=4P(u(t))-C \l( \int_{|x|>R}  |\del_x u(t,x)|^2 dx+\int_{|x|>R}  |u(t,x)|^2 dx+\int_{|x|>R}  |u(t,x)|^{p+1} dx \r).
\end{align*}
We note that there exists $\delta>0$ independent of $t$ such that $P(u(t))>\delta$ by \propref[prop2.18] since $u$ belongs to $\cM_{\omega}^{+}$. Therefore, by \lemref[lem3.12], if we take $\eps \in (0,3\delta)$, then there exists $R>0$ such that $I''(t)\geq \delta$ for any $t\in\R_+$. On the other hand, the mass conservation laws gives $I(t)\leq R^2 \norm[u(t)]_{L^2}^2<C$, where $C$ is independent of $t$, for any $t\in\R_+$. Hence, we obtain a contradiction. 


\section{Proof of the blow-up part}

To prove the blow-up results, we use the method of Du--Wu--Zhang \cite{DWZ}. On the contrary, we assume that the solution $u$ to \eqref[deltaNLS] with $u_0 \in \cM_\omega^{-}$ is global in the positive time direction and $\sup_{t\in\R_+}\norm[\del_x u(t)]_{L^2}^2<C_0<\infty$. Then, we have $\sup_{t\in \R_+} \norm[u(t)]_{L^q} <\infty $ for any $q>p+1$ by the energy conservation and the Sobolev embedding.

For $R>0$, w e take $\phi$ such that
\begin{align*}
&\phi(r)=\l\{
\begin{array}{ll}
0, & 0<r<R/2,
\\
1, & r\geq R,
\end{array}\r.
\\
&0\leq \phi\leq 1, \qquad \phi'\leq 4/R.
\end{align*}

By the fundamental formula and the H\"{o}lder ineqaulity, we have 
\begin{align*}
I(t)
&=I(0)+\int_{0}^{t} I'(s)ds
\leq I(0) +\int_{0}^{t} | I'(s)|ds
\\
&\leq I(0) +  t \norm[\phi']_{L^\infty} \norm[u(t)]_{L^2}^2 \norm[\del_x u(t)]_{L^2}^2
\\
&\leq I(0) +\frac{8 M(u) C_0 t}{R}.
\end{align*}

Here, we note that $ I(0) \leq \int_{|x|> R/2} |u(0,x)|^2 dx=o_R(1)$ and $\int_{|x|>R} |u(t,x)|^2 dx \leq I(t)$. 
Therefore,  we obtain the following lemma.
\begin{lemma} \label{lem4.1}
Let $\eta_0 >0 $ be fixed. Then, for any $t\leq \eta_0 R /(8M(u) C_0)$, we have
\[ \int_{|x|>R} |u(t,x)|^2 dx \leq o_R(1) + \eta_0. \]
\end{lemma}

We take another $\phi$ such that 
\[ 0\leq \phi \leq r^2, \quad | \phi' | \cleq r, \quad |\phi'' | \leq 2, \quad |\phi^{(4)}| \leq \frac{4}{R^2}, \]
and 
\[ \phi (r)=\l\{ \begin{array}{ll} r^2, & 0\leq r  \leq R, \\ 0,& r\geq 2R. \end{array} \r.\]

Then we have the following lemma. 
\begin{lemma}
There exist two constants $C=C(p,M(u),C_0)>0$ and $\theta_q>0$ such that
\[  I''(t) \leq 4P(u(t))+C\norm[u]_{L^2(|x|>R)}^{\theta_q}+ C R^{-2} \norm[u]_{L^2(|x|>R)}^{2}. \]
\end{lemma}

\begin{proof}
By \eqref[eq3.3],  we have
\begin{align*}
I''(t) 
= 4P(u(t))+R_1+R_2+R_3.
\end{align*}
At first, we prove $R_1\leq 0$. By the definiton of $\phi$, we see that
\[ R_1=\re \int_{\R} \{ \del_x^2(\phi(|x|))-2\} |\del_x u(t,x)|^2 dx=\re \int_{\R} \{  \phi''(|x|)-2\} |\del_x u(t,x)|^2 dx \leq 0.\]
At second, we consider $R_2$. By the H\"{o}lder inequality, we have
\begin{align*}
R_2&=- \frac{p-1}{p+1} \re \int_{\R} \{ \del_x^2(\phi(|x|))-2\}  |u(t,x)|^{p+1} dx 
\\
&\leq C \int_{|x|>R}  |u(t,x)|^{p+1} dx 
\\
&\leq C \norm[u]_{L^q(|x|>R)}^{1-\theta_q} \norm[u]_{L^2(|x|>R)}^{\theta_q}
\\
&\leq C  \norm[u]_{L^2(|x|>R)}^{\theta_q},
\end{align*}
where $q>p+1$ and $0<\theta_q \leq 1$, since $\sup_{t\in \R_+} \norm[u(t)]_{L^q} <\infty $.
At third, we consider $R_3$. 
\begin{align*}
R_3&=- \frac{1}{4} \re \int_{\R} \del_x^4 (\phi(|x|)) |u(t,x)|^2 dx
\\
&\leq C R^{-2} \int_{|x|>R}  |u(t,x)|^2 dx= C R^{-2} \norm[u]_{L^2(|x|>R)}^{2}.
\end{align*}
Therefore, we complete the proof.
\end{proof}

\begin{proof}[Proof of (2) in main theorems]
Since $u(t)$ belongs to $\cM_\omega^{-}$, there exists $\delta>0$ independent of $t$ such that $P(u(t))<-\delta$ for all $t\in\R_+$ by \propref[prop2.18]. Therefore, we obtain
\[ I''(t) \leq -4\delta +C\norm[u]_{L^2(|x|>R)}^{\theta_q}+ C R^{-2} \norm[u]_{L^2(|x|>R)}^{2}. \]
We take $\eta_0>0$ such that $C\eta_0^{\theta_q}+C\eta_0^{2}<\delta$. By \lemref[lem4.1], for $t\in [0, \eta_0 R /(8M(u) C_0)]$, we have
\[ I''(t) \leq -3\delta +o_R(1). \]
Let $T:=\eta_0 R /(8M(u) C_0)$. Integrating the above inequality from $0$ to $T$, we get
\[ I(T)\leq I(0)+I'(0)T +\frac{1}{2}(-3\delta +o_R(1))T^2.\]
For sufficiently large $R>0$, we have $-3\delta+o_R(1)<-2\delta$. Thus,  we get
\[ I(T)\leq I(0)+I'(0)\eta_0 R /(8M(u) C_0) -\alpha_0 R^2,\]
where $\alpha_0:=\delta \eta_0^2 /(8M(u) C_0)^2>0$.
And we can prove $I(0)=o_R(1)R^2$ and $I'(0)=o_R(1)R$. Indeed, 
\begin{align*}
I(0) & \leq \int_{|x|<\sqrt{R}} |x|^2 |u_0(x)|^2 dx+ \int_{\sqrt{R}<|x|<2R} |x|^2 |u_0(x)|^2 dx
\\
& \cleq M(u)R +R^2 \int_{\sqrt{R}<|x|}  |u_0(x)|^2 dx
\\
&=o_R(1)R^2,
\end{align*}
and 
\begin{align*}
I'(0) & \leq \int_{|x|<\sqrt{R}} |\phi'(|x|)| |u_0(x)||\del_x u_0 (x)| dx+ \int_{\sqrt{R}<|x|<2R} |\phi'(|x|)| |u_0(x)||\del_x u_0 (x)| dx
\\
& \leq \int_{|x|<\sqrt{R}} |x| |u_0(x)||\del_x u_0 (x)| dx+ \int_{\sqrt{R}<|x|<2R} |x| |u_0(x)||\del_x u_0 (x)| dx
\\ 
& \cleq \norm[u_0]_{H^1}^2 \sqrt{R} +R \int_{\sqrt{R}<|x|} |u_0(x)||\del_x u_0 (x)| dx
\\
&=o_R(1)R.
\end{align*}
Therefore, we see that 
\[ I(T)\leq o_R(1)R^2 -\alpha_0 R^2.\]
For sufficiently large $R>0$, $o_R(1) -\alpha_0 <0$. However, this contradicts $I(T)=\int_{\R} \phi(|x|) |u(T,x)|^2 dx>0$. This argument can be applied in the negative time direction. 
\end{proof}



\appendix


\section{Rewrite Main Theorem into the version independent of the frequency}

We prove Corollary 1.4. To see this, it is sufficient to prove the following lemma. 

\begin{lemma}
Let $\varphi \in H^1(\R)$. The following statements are equivalent.
\begin{enumerate}
\item There exists $\omega >0$ such that $S_{\omega} (\varphi) < l_{\omega}=n_{\omega}$.
\item $\varphi$ satisfies $E(\varphi)M(\varphi)^{\sigma} < E_0(Q_{1,0}) M(Q_{1,0})^{\sigma} $. 
\end{enumerate}
\end{lemma}

\begin{proof}
If $\varphi=0$, the statement holds. 
Let $\varphi \in H^1(\R) \setminus \{0\}$ be fixed. We define $f(\omega):=l_{\omega}-S_{\omega}(\varphi)$. Then, (1) is true if and only if $\sup_{\omega>0} f(\omega)>0$. Noting that $l_{\omega}=\omega ^{\frac{p+3}{2(p-1)}}S_{1,0}(Q_{1,0})$, $f$ is muximum at $\omega= \omega_0$ where 
\[ \omega_0 :=\l( \frac{M(\varphi)}{ \frac{p+3}{2(p-1)}S_{1,0}(Q_{1,0})}  \r)^{-\frac{2(p-1)}{p-5}} >0. \]
Therefore, (1) is equivalent to $f(\omega_0)>0$. Now, since
\begin{align*}
f(\omega_0)
&= \l( \frac{M(\varphi)}{ \frac{p+3}{2(p-1)}S_{1,0}(Q_{1,0})}  \r)^{-\frac{p+3}{p-5}} S_{1,0}(Q_{1,0}) 
-  \l( \frac{M(\varphi)}{ \frac{p+3}{2(p-1)}S_{1,0}(Q_{1,0})}  \r)^{-\frac{2(p-1)}{p-5}} M(\varphi) -E(\varphi)
\\
&=\frac{ \l( \frac{p+3}{2(p-1)} S_{1,0}(Q_{1,0}) \r) ^{\frac{2(p-1)}{p-5}}}{M(\varphi)^{\frac{p+3}{p-5}}} -E(\varphi)
>0,
\end{align*}
we have $\l( \frac{p+3}{2(p-1)} S_{1,0}(Q_{1,0}) \r) ^{\frac{2(p-1)}{p-5}} >E(\varphi)M(\varphi)^{\frac{p+3}{p-5}}$. Noting $Q_{1,0}$ satisfies 
\[  \norm[Q_{1,0}]_{L^2}^2 = \frac{p+3}{2(p-1)} \norm[\del_x Q_{1,0}]_{L^2}^2= \frac{p+3}{2(p+1)} \norm[Q_{1,0}]_{L^{p+1}}^{p+1}, \]
we have $\l( \frac{p+3}{2(p-1)} S_{1,0}(Q_{1,0}) \r) ^{\frac{2(p-1)}{p-5}}=E_0(Q_{1,0})M(Q_{1,0})^{\frac{p+3}{p-5}}$. This completes the proof. 
\end{proof}

\section*{Acknowledgments} The authors would like to express deep appreciation to Professor Kenji Nakanishi for many useful suggestions, comments and constant encouragement. The first author is supported by Grant-in-Aid for JSPS Fellows 26$\cdot$1884 and Grant-in-Aid for Young Scientists (B) 15K17571. The second author is supported by Grant-in-Aid for JSPS Fellows 15J02570. 


\end{document}